%% file: coherent.tex
\UseRawInputEncoding
\documentclass[10pt,notitlepage]{amsart}
\input{macros}
\usepackage{xparse}

\renewcommand{\mod}{\text{\rm{-mod}}}

\DeclareMathOperator{\DRng}{DRng}

\newcommand{\R}{\mathscr{R}}

\newcommand{\Cplt}{\mathbf{Cplt}}
\newcommand{\Reg}{\mathbf{Reg}}
\newcommand{\Coht}{\mathbf{Coh}}

\DeclareMathOperator{\Cohp}{{{\widehat\Coh}_+}}
\DeclareMathOperator{\DCohp}{{{\widehat\DCoh}_+}}
\DeclareMathOperator{\DCohm}{{{\widehat\DCoh}_{--}}}
\DeclareMathOperator{\DCohpm}{{\widehat{\DCoh}_\pm}}
\DeclareMathOperator{\Dualiz}{Dualiz}

\newcommand{\proper}{\mathit{prop}}

\newcommand{\bt}{\beta}

\makeatletter
\newcommand{\Free@withargs}[1]{\mathrm{Free}^{{\mkern-2.5mu\scriptstyle{\mathrm{#1}}}}}
\newcommand{\Free@noargs}{\mathrm{Free}}
\DeclareDocumentCommand\Free { g }{%
    \IfNoValueT{#1}{\Free@noargs{}}%
    \IfNoValueF{#1}{\Free@withargs{#1}}%
}
\newcommand{\coFree@withargs}[1]{\mathrm{coFree}^{{\mkern-2.5mu\scriptstyle{\mathrm{#1}}}}}
\newcommand{\coFree@noargs}{\mathrm{coFree}}
\DeclareDocumentCommand\coFree { g }{%
    \IfNoValueT{#1}{\coFree@noargs{}}%
    \IfNoValueF{#1}{\coFree@withargs{#1}}%
}

\begin{document}

\makeatother

\title{Integral transforms for coherent sheaves}

\author{David Ben-Zvi} \address{Department of Mathematics\\University
  of Texas\\Austin, TX 78712-0257} \email{benzvi@math.utexas.edu}
\author{David Nadler} \address{Department of Mathematics\\University
  of California\\Berkeley, CA 94720-3840}
\email{nadler@math.berkeley.edu}
\author{Anatoly Preygel} 
\email{anatoly.preygel@gmail.com}

\begin{abstract}

The theory of integral, or Fourier-Mukai, transforms between derived categories of sheaves is a well established tool in noncommutative algebraic geometry. 
General ``kernel theorems" represent all reasonable linear functors
between categories of perfect complexes (or their ``large" version, quasi-coherent complexes) on schemes and stacks over some fixed base 
as integral kernels in the form of 
 complexes (of the same nature) on the fiber product.
However, for many applications in mirror symmetry and 
geometric representation theory one is interested instead in the bounded derived category of coherent sheaves (or its ``large" version, ind-coherent sheaves), 
which differs from perfect complexes (and quasi-coherent complexes) once the underlying variety is singular. 
In this paper, we prove general kernel theorems for linear functors between derived categories of coherent sheaves over a base
in terms of integral kernels on the fiber product.
Namely, we identify coherent kernels with functors taking perfect complexes to coherent complexes (an analogue of the 
classical Schwartz kernel theorem), and 
kernels which are coherent relative to the source with functors taking all coherent complexes to coherent complexes.
The proofs rely on key aspects of the ``functional analysis" of derived categories, namely the distinction between small and large
categories and its measurement using $t$-structures. These are used in particular to 
correct the failure of integral transforms on ind-coherent complexes to correspond to ind-coherent complexes on a fiber product.
The results are applied in a companion paper to the representation theory of the affine Hecke category, identifying 
affine character sheaves with the spectral geometric Langlands category in genus one.
\end{abstract}

\maketitle

\tableofcontents


\section{Introduction}

\subsection{Coherent sheaves and linear functors}

Integral transforms on derived categories of quasicoherent sheaves
have been intensely studied since Mukai introduced his analogue of the
Fourier transform for abelian varieties. The results of \cites{Orlov, BLL, Toen, BFN, LO}, 
a far from comprehensive but we hope nevertheless useful list of references,
give 
increasingly strong statements to the effect that all reasonable functors between derived categories of
perfect complexes or quasicoherent sheaves can be represented by integral kernels, once one works 
 in the appropriate homotopical and geometric  settings. (There are also recent developments
 devoted to the subtle behavior of functors between traditional as opposed to enhanced derived categories~\cites{Neeman, 
 RvdB}.)
 In this paper we extend this theory to coherent sheaves
 on singular varieties and stacks, in particular providing an analogue of the classical Schwartz kernel theorem in which 
perfect complexes play the role of test functions and coherent sheaves that of distributions. 
In a sequel \cite{BNP2}, these results are applied to calculate the categorical Hochschild invariants 
of the affine Hecke category. This is generalized in~\cite{BNglue} to derive a gluing paradigm
for coherent sheaves on character stacks.
 
\begin{remark}[Standing assumptions throughout the paper] Henceforth category will stand for pre-triangulated $k$-linear
 dg category or stable $k$-linear $\infty$-category, where $k$ is a field of characteristic zero. An important convention: we will
 use  homological grading (denoted by subscripts) for chain complexes and $t$-structures instead of the prevalent cohomological grading (denoted by superscripts).  
 
 Henceforth schemes, algebraic spaces and stacks are all  over $k$ and assumed to be derived unless explicitly called ``classical.''  We write $X_{cl}$ for the  classical scheme, algebraic space or stack underlying a given $X$.
\end{remark}

\begin{remark}[Standing assumptions throughout the introduction] For the purposes of the introduction,
 all schemes, algebraic spaces and stacks will be quasi-compact  and almost of finite presentation over $k$.   To simplify the  discussion outside of the formal statements of theorems,  we often assume all schemes, algebraic spaces and stacks are  geometric and perfect in the sense of \cite{BFN}.  Recall that if a stack $X$ is perfect, then
we can recover the ``large'' (presentable, in particular cocomplete) 
 category of quasicoherent sheaves $\QC(X)$
 and the ``small'' (small, idempotent-complete) category of perfect complexes $\Perf(X)$ from one another.  Namely, we recover $\QC(X) = \Ind \Perf(X)$ by ind-completing and $\Perf(X) = \QC(X)^{c}$ by taking compact objects.
All quasi-compact and separated algebraic spaces, all smooth geometric finitely-presented $k$-stacks, and most commonly occurring stacks in characteristic zero give examples of perfect stacks.
\end{remark}

Let $p_X:X\to S$, $p_Y:Y\to S$ be maps of perfect stacks.
Then \cite{BFN}*{Theorem 1.2} asserts that all $\QC(S)$-linear functors between quasicoherent sheaves on $X$ and $Y$ are represented 
by integral transforms:
\begin{equation}\label{eq qc fun equiv} 
\xymatrix{
\Phi_{(-)}:\QC(X \times_S Y) \ar[r]^-\sim & \Fun^L_{\QC(S)}(\QC(X), \QC(Y)) &
\Phi_\K( - ) = p_{Y*}(p_X^*(-) \otimes \K)
}
\end{equation}
Equivalence~(\ref{eq qc fun equiv}) is established by first showing that the external tensor product descends to an equivalence
\begin{equation}\label{eq qc tens equiv}
\xymatrix{
\QC(X)\otimes_{\QC(S)} \QC(Y) \ar[r]^-\sim &
\QC(X \times_S Y) &
}
\end{equation}
and then showing that $\QC(X)$ is self-dual as a $\QC(S)$-module.

Now let us focus on  small categories of perfect complexes.
By definition of the tensor product, equivalence~(\ref{eq qc tens equiv}) restricts to an equivalence on compact objects
\begin{equation}\label{eq perf tens equiv}
\xymatrix{
\Perf(X)\otimes_{\Perf(S)} \Perf(Y) \ar[r]^-\sim &
\Perf(X \times_S Y) &
}
\end{equation}
In contrast, the integral transform $\Phi_\P$ associated to a perfect kernel $\P \in \Perf (X \times_S Y)$ will not in general take perfect complexes to perfect complexes.  But if we assume that $p_X:X\to S$ is smooth and proper, then equivalence~(\ref{eq qc fun equiv}) restricts to an equivalence
\begin{equation}\label{eq perf fun equiv}
\xymatrix{
\Phi_{(-)}:\Perf(X \times_S Y) \ar[r]^-\sim & \Fun^{ex}_{\Perf(S)}(\Perf(X), \Perf(Y))
}
\end{equation}

On a singular stack, there are more bounded coherent complexes than perfect complexes, and they form
an intermediary\footnote{Recall that within the introduction, our standing assumptions imply that $X$ has finite Tor-dimension over $k$, so that perfect complexes are indeed coherent.} small stable category $\Perf(X)\subset \DCoh(X)\subset \QC(X)$. Here and throughout we write $\DCoh(X)$ for the enhanced analogue of the classical
bounded derived category $D^b(X)$.
Now suppose that $p_X:X\to S$ is proper  
but not necessarily smooth. Hence the fiber product $X \times_S Y$ is potentially singular and so carries more coherent than perfect complexes.

The first  goal of this paper is to answer the following natural question
 (with applications discussed below):
 \begin{center}{\em What kind of linear functors are given by coherent integral kernels?} \end{center}
The following theorem, the subject of \autoref{sec:perf}, shows that they give linear functors
on perfect complexes
with coherent as opposed to perfect values.

\begin{theorem} \label{intro thm int trans}
Let $S$ be a perfect stack, $p_X \colon X\to S$ a proper relative algebraic space, and $Y$ a locally Noetherian $S$-stack.

Then the integral transform construction provides an equivalence
  $$
  \xymatrix{
  \Phi : \DCoh(X \times_S Y) \ar[r]^-\sim &   \Fun^{ex}_{\Perf S}(\Perf X , \DCoh Y)
}  $$
\end{theorem}

\begin{remark}
The assumption that $p_X$ be proper can be weakened to separated by changing the source of $\Phi$ to consist
of those coherent complexes whose support is proper over $Y$ (see \autoref{thm: dcoh shriek}).
\end{remark}

\begin{remark} There is the following useful mnemonic for \autoref{intro thm int trans}.  
By analogy with ordinary commutative rings, when $p_X:X\to S$ is proper we could write $p_X^! = \Fun^{ex}_{\Perf S}(\Perf X, -)$ and 
think of it as a $!$-pullback, in contrast to the $*$-pullback $p_X^* = \Perf X \otimes_{\Perf S} (-)$.  
By equivalence (\ref{eq perf tens equiv}), we know that  $\Perf$ forms a presheaf under $p_X^*$, 
while \autoref{intro thm int trans} says that $\DCoh$ forms a presheaf under $p_X^!$.
\end{remark}

\begin{remark}\label{perf thm corollaries} Let us highlight some special cases:
  \begin{enumerate}
  \item  If $X \to S$ is \emph{smooth} and proper, then $\Perf X$ is self-dual over $\Perf S$ so that
  \[ \Fun_{\Perf S}^{ex}(\Perf X, \DCoh Y) \isom \Perf X \otimes_{\Perf S} \DCoh Y. \]
  It is known \cites{indcoh} that in this case the exterior tensor product induces an equivalence
  \[ \Perf X \otimes_{\Perf S} \DCoh Y \stackrel\sim\longrightarrow \DCoh(X \times_S Y) \]
  thereby recovering the theorem.

\item Consider the case of $Y = S$, with $p_X:X\to S$ proper. Then the theorem states that bounded linear functionals on perfect complexes are given by coherent complexes 
$$
  \xymatrix{
  \Phi : \DCoh(X) \ar[r]^-\sim &   \Fun^{ex}_{\Perf S}(\Perf X , \DCoh S)
}  $$

\item Suppose that $Y$ is regular so that $\DCoh(Y) \simeq \Perf(Y)$. Then the theorem states that linear functors on perfect complexes are given by integral transforms with coherent kernels
  $$
  \xymatrix{
  \Phi : \DCoh(X \times_S Y) \ar[r]^-\sim &   \Fun^{ex}_{\Perf S}(\Perf X , \Perf Y )
}  $$

\end{enumerate}
\end{remark}

\begin{remark}\label{Schwartz}[Schwartz Kernel Theorem] The identification $\DCoh(X)\simeq \Fun^{ex}_{\Perf S}(\Perf X , \DCoh S)$ for $X/S$ proper
supports an interpretation of perfect complexes as algebraic analogues of test functions and 
of coherent complexes as distributions. In this interpretation, Theorem~\ref{intro thm int trans} becomes (a relative version of)
the Schwartz kernel theorem, identifying continuous linear operators, from test functions on a manifold $X$ to distributions on another manifold
$Y$, with distributions on $X\times Y$.
\end{remark}


\subsection{Functors out of $\DCoh$}

Our second (and significantly more involved) main theorem provides a counterpart for \autoref{intro thm int trans} identifying functors out of categories of coherent sheaves as integral
kernels that are coherent {\em relative to the source}. 
In order to formulate this notion, we need to recall the notion of an \emph{almost perfect complex} \cite{LurieHA}; this is closely related to the classical notion of a {\em pseudo-coherent sheaf} \cite{illusie}.

\begin{defn} Suppose $X$ is Noetherian. Define $\DCohp(X)\subset \QC(X)$ to be the full subcategory   consisting of (homologically) bounded below complexes whose homology sheaves are coherent as $\O_{X_{cl}}$-modules.
\end{defn}

\begin{remark}

The notation is suggested by the fact that $\DCohp(X)$ is the left completion of $\DCoh(X)$ with respect to the standard $t$-structure.  

The objects of $\DCohp(X)$ admit another description, in line with the classical notion of a pseudo-coherent sheaf or Lurie's almost perfect complex. Recall that $\F \in \QC(X)$ is almost perfect if and only if $\tau_{\leq n} \F$ is a compact object of $\QC(X)_{\leq n}$, for all $n$.  If $X$ is Noetherian, one checks that this coincides with the above characterization.
\end{remark}

By construction, there is an inclusion $\DCoh(X) \subset \DCohp(X)$.  We can characterize the objects of $\DCoh(X)$ as those that are $t$-bounded in both directions, or alternatively as those objects of $\DCohp(X)$ that have finite Tor-dimension over the base field $k$.  This motivates the following definition:

\begin{defn} For $X\to S$, define the full subcategory category $\DCoh(X/S) \subset \DCohp(X)$ of complexes on $X$ that are {\em coherent relative to $S$} to
consist of complexes with finite Tor-dimension with respect to $S$.
\end{defn}

Thus we have $\DCoh(X/k)=\DCoh(X)$, while $\DCoh(X/X) = \Perf(X)$ by the well-known characterization of perfect complexes as those almost perfect complexes of finite Tor-dimension. Categories of relative coherent sheaves are used in \cite{lowrey} to define moduli stacks of objects of $\DCoh(X)$.

Now the following theorem is the main result of \autoref{sect: DCoh}:

\begin{theorem}\label{theorem:intro-fun-dcoh}
Let $S$ be a quasi-compact, geometric, smooth $k$-stack.  Let $p_X\colon X \to S$ be a proper relative $S$-algebraic space locally of finite presentation, and $p_Y\colon Y \to S$ a locally finitely presented $S$-stack.  Then the integral transform construction gives an equivalence
  \[ \Phi \colon 
  \DCoh(X \times_S Y/ X) \stackrel\sim\longrightarrow \Fun^{ex}_{\Perf S}(\DCoh X, \DCoh Y)  \]
 \end{theorem}

\begin{remark}
The assumption that $p_X$ is proper can be weakened to separated if we change the source of $\Phi$ to consist
of kernels with support proper over $Y$.
\end{remark}

\begin{remark}\label{reflexive} Let us highlight an interesting special case of the theorem. 
Suppose $p \colon X \to S$ is a proper relative algebraic space, and that $S$ is a regular Noetherian stack (so that $\DCoh (S)=\Perf (S)$). 
Then \autoref{theorem:intro-fun-dcoh} and 
\autoref{perf thm corollaries} give the following ``dual" statements characterizing $\DCoh (X)$ and $\Perf (X)$ as bounded linear functionals on each other:
 \[ \Fun^{ex}_{\Perf S}(\DCoh X, \Perf S) \isom \Perf (X) \]
  \[ \Fun^{ex}_{\Perf S}(\Perf X, \Perf S) \isom \DCoh (X) \]
We could summarize this by saying that $\Perf (X)$ and $\DCoh (X)$ are {\em reflexive} $\Perf (S)$-linear categories.  Notice that they are not actually dualizable unless $X \to S$ is also smooth!

This weak duality between $\DCoh$ and $\Perf$ stands in striking contrast
to the situation with their ``large versions": the categories $\QC(X)=\Ind\Perf (X)$ and $\QCsh(X)=\Ind\DCoh (X)$ are each 
{\em self-dual}. Another divergence between the large and small worlds is highlighted in the next section.
\end{remark}


\subsection{Measuring categories}

We conclude the introduction by highlighting an aspect of the ``functional analysis" of categories that 
figures prominently in the statement and proof
of \autoref{intro thm int trans} and especially \autoref{theorem:intro-fun-dcoh}: the use of $t$-structures to modify ``growth properties" of objects, and 
thereby account for the distinction between small categories and their large cocomplete versions.
More specifically, in~\autoref{intro thm regularize} below, we describe how $t$-exactness properties of functors
on ind-coherent sheaves correspond to $t$-boundedness properties of their integral kernels.

Recall that for $X$ a perfect  stack,
the cocompletion of the  category $\Perf(X)$ of perfect complexes is the category  of quasicoherent complexes $\QC(X)\simeq \Ind\Perf(X)$. 
The cocompletion of the  category $\DCoh(X)$ of coherent complexes is the category  of ind-coherent complexes $\QCsh(X) \simeq \Ind\DCoh(X)$.  Each of these large categories carries a natural $t$-structure.

However, observe that the equivalence 
\[ \DCoh(X \times_S Y) \isom   \Fun^{ex}_{\Perf S}(\Perf X , \DCoh Y)\]
of \autoref{intro thm int trans} fails if we replace $\Perf$ by $\QC$ and $\DCoh$ by $\QC^!$.

To see this, let us place $\QC^!(X\times_S Y) = \Ind \DCoh(X \times_S Y)$ on the left hand side.
On the right hand side, since $\QC(X)$ is self-dual over $\QC(S)$, we find the identification
$$
\xymatrix{
  \Fun^{L}_{\QC(S)}(\QC(X) , \QC^!(Y)) \simeq \QC(X)\otimes_{\QC(S)} \QC^!(Y) 
  }
$$
Let us further assume that $X$ and $Y$ are smooth, so that $\QC^!(Y) \simeq \QC(Y)$, then by equivalence~(\ref{eq qc tens equiv}), we obtain a further equivalence
$$
\xymatrix{
 \Fun^{L}_{\QC(S)}\left(\QC(X)=\QCsh(X) , \QC(Y)=\QCsh(Y)\right) \simeq \QC(X\times_S Y)
}
$$
But with the current setup, $X \times_S Y$ need not be smooth, and hence in general $\QC(X\times_S Y)$ will not be equivalent
to $\QC^!(X \times_S Y)$. 
This precise setup arises in our motivating case of the affine Hecke category \cite{BNP2}.

We may look at this discrepancy another way. For quasi-coherent complexes, it is often the case that categories of functors coincide with quasi-coherent complexes on the fiber product. For ind-coherent complexes, this is often true over a point but very rarely true over a non-trivial base, as we have just seen even when source and target are smooth.  

The goal of \autoref{sect: shriek} is to ``fix'' this, or rather to show that the failure of the integral transform construction to give an equivalence can be precisely controlled by means of the $t$-structure.   More precisely, in \autoref{thm:fun-regulariz}, we show:

\begin{theorem} \label{intro thm regularize} Suppose that $S$ is a quasi-compact, geometric, finitely-presented $k$-stack; that $p \colon X \to S$ is a quasi-compact and separated $S$-algebraic space of finite-presentation; and, that $Y$ is an $S$-stack of finite presentation over $S$.  Then the $!$-integral transform 
  \[ \Phish\colon \QCsh(X \times_S Y) \longrightarrow \Fun^L_{\QCsh (S)}(\QCsh (X), \QCsh (Y)) \]
  need not be equivalence, but it \emph{does} induce an equivalence between the bounded-above objects and those functors which are left $t$-exact up to a shift:
  \[ 
  \Phish\colon \QCsh(X \times_S Y)_{<\infty} \stackrel{\sim}{\longrightarrow}  \left\{ F \in  \Fun^L_{\QCsh S}(\QCsh X, \QCsh Y) \colon \begin{gathered} F((\QCsh X)_{<0}) \subset (\QCsh X)_{<N(F)} \\ \text{ for some $N(F)$ depending on $F$} \end{gathered} \right\}
   \] 
\end{theorem}

This is an essential ingredient in proof of \autoref{theorem:intro-fun-dcoh}, in effect reducing us to proving that any functor for the small categories is automatically left $t$-exact up to a shift.  

In the special case that $S$ is smooth, we can make this functional analysis of $t$-structures even more precise.  The appendix (\autoref{app:t-bdd}) provides a discussion of operations on categories with reasonably behaved $t$-structures.  Such a category $\C$ comes equipped with a stable subcategory $\Coh(\C)$ of ``coherent'' objects, 
and the {\em regularization} of $\C$ is the corresponding ind-coherent category \[ \R(\C) \eqdef \Ind\Coh(\C)\to \C \]
A dual notion to regularization is the (left) $t$-completion of $\C$, which is the limit 
\[ \C\to \oh{\C} \eqdef \ilim \C_{<n} \]  If we restrict ourselves to considering $t$-exact functors, the theories of complete and regular categories with $t$-structure are \emph{equivalent} (see \autoref{thm:reg-coh-equiv}).   

For a geometric stack of finite type
over $k$, the natural functor $\QCsh(X)\to \QC(X)$ presents $\QCsh(X)$ as the regularization of $\QC(X)$, and $\QC(X)$ as the completion of $\QCsh(X)$.   Thus in cases where a reasonable theorem -- involving only functors left $t$-exact up to a shift -- holds for $\QC$ but not $\QCsh$, we can expect that we should be able to fix this defect by regularizing.  

\begin{theorem} Suppose $S$ is a quasi-compact, geometric, smooth $k$-stack; that $p \colon X \to S$ is a quasi-compact and separated relative $S$-algebraic space of finite-presentation; and that $Y$ is an $S$-stack of finite presentation.
 
Then there is a $t$-structure on the functor category so that
  \[ \Fun^L_{\QCsh (S)}(\QCsh (X), \QCsh (Y))_{\leq 0} = \left\{ \text{left $t$-exact functors} \right\} \]
and the $!$-integral transform
  \[ \Phish \colon \QCsh(X \times_S Y) \longrightarrow \Fun^L_{\QCsh (S)}(\QCsh (X), \QCsh (Y)) \] is left $t$-exact up to a shift and  exhibits $\QCsh(X \times_S Y)$ as 
a regularization of $\Fun^L_{\QCsh S}(\QCsh X, \QCsh Y)$.  (In particular, it induces an equivalence on bounded above objects, yielding a special case of \autoref{intro thm regularize}.)
\end{theorem}

\subsection{Acknowledgements}
We gratefully acknowledge the support of NSF grants DMS-1103525 (DBZ), DMS-1319287 (DN),
and an NSF Postdoctoral Fellowship (AP).

\section{Preliminaries}
We adopt the functor of points viewpoint.  Let $\DRng$ denote the $\infty$-category of connective $E_\infty$-algebras.  A pre-stack will be a functor 
$\DRng^{\op} \to \sSet$.  For any such, we may define $\QC(\X)$ by Kan extension from the the case of affines.  
For $\pi \colon \X \to \Y$ a map of prestacks, there is a pullback functor $\pi^*$ 
defined by restriction of indexing diagram -- we define $\pi_*$ to be the right adjoint to this functor, which exists by general nonsense on presentable $\infty$-categories.


As is well-known, pushforwards for arbitrary quasi-coherent complexes on stacks are problematic.  However, for (homologically) bounded above complexes this is not an issue:
\begin{lemma}\label{lem:push-bdd-above} Suppose that $\pi\colon \X \to S$ is a quasi-compact and quasi-separated morphism of stacks.  Then,
  \begin{enumerate}
      \item For any (homologically) bounded above object $\F \in \QC(\X)_{<\infty}$, the base-change formula holds with respect to maps of finite Tor dimension.
      \item $R\pi_*$ preserves filtered colimits (equivalently, infinite sums) on (homologically) uniformly bounded above objects.
    \end{enumerate}
\end{lemma}

If $\X$ is assumed to be of \emph{finite (quasi-coherent) cohomological dimension} then this is not an issue:
\begin{prop}\label{prop:push-CD} Suppose that $\pi \colon \X \to S$ is a quasi-compact and quasi-separated morphism of stacks, and that $\pi$ has quasi-coherent cohomological dimension universally bounded by $d$ i.e., that there exists an integer $d$ such that for any base-change of $\pi$ we have $\pi_* \F \in \QC(S)_{>-d}$ for $\F \in \QC(\X)_{>0}$.  Then:
  \begin{enumerate}
      \item $\pi_*\colon \QC(\X) \to \QC(S)$ preserves filtered colimits;
      \item $\pi_*$ and $\pi^*$ satisfy the projection formula;
      \item the formation of $\pi_*$ is compatible with arbitrary base-change.
    \end{enumerate}
\end{prop}
\begin{proof}[Sketch] The proof of (i)-(iii) is via the bounded case and convergence for each homology sheaf by the boundedness of cohomological dimension,  c.f. \cite{DrinfeldGaitsgory}
\end{proof}

Any morphism which is a quasi-compact qusi-separated relative algebraic space, and many stacky maps in characteristic zero by by \cite{DrinfeldGaitsgory}, satisfy the hypothesis of the previous proposition.  Thus the previous proposition will apply to every pushforward we take in this paper.


\section{Functors out of $\Perf$}\label{sec:perf}
\begin{na} The standing assumptions for this section, unless otherwise stated, are: $S$ is a perfect derived stack; $X \to S$ is a quasi-compact and separated (derived) $S$-algebraic space locally of finite presentation; $Y \to S$ is a locally Noetherian (derived) $S$-stack.
\end{na}

Let $p_X \colon X\to S$, $p_Y\colon Y \to S$ be maps of derived stacks, and suppose that $X$ and $S$ are perfect. Recall that linear functors for quasi-coherent complexes are given by $*$-integral transforms
$$
\xymatrix{
\Phi\colon\QC(X \times_S Y) \ar[r]^-\sim & \Fun^L_{\QC(S)}(\QC(X), \QC(Y)) &
\Phi_\K(\F)= p_{Y*}(p_X^*(\F) \otimes \K)
}
$$
Since $X$ and $S$ are perfect, there is a natural induction equivalence
$$
\xymatrix{
\Fun^{ex}_{\Perf S}(\Perf X, \QC(Y)) \ar[r]^-\sim & \Fun^L_{\QC(S)}(\QC(X), \QC(Y)) 
}$$
Thus we could reformulate the above as an equivalence
$$
\xymatrix{
\Phi\colon\QC(X \times_S Y) \ar[r]^-\sim & \Fun^{ex}_{\Perf(S)}(\Perf(X), \QC(Y)) &
}
$$

Let us restrict to the full subcategory of integral kernels
$$
\DCoh_{\proper/Y}(X \times_S Y) \subset \QC(X \times_S Y)
$$
that are coherent with  support proper over $Y$.

The following is the main result of this subsection.

\begin{theorem}\label{thm: dcoh shriek}
Suppose that $S$ is a perfect stack; that $p_X \colon X \to S$ is a quasi-compact and separated $S$-algebraic space locally of finite presentation; and, that $Y$ is a locally Noetherian $S$-stack.

The $*$-integral transform construction provides an equivalence
  $$
  \xymatrix{
  \Phi \colon \DCoh_{\proper/Y}(X \times_S Y) \ar[r]^-\sim &   \Fun^{ex}_{\Perf S}(\Perf X , \DCoh Y)
}  $$

\end{theorem}

\begin{corollary}\label{cor: dcoh shriek}
 Suppose that $X$, $S$, and $Y$ are as in the previous Theorem, and furthermore that $Y$ is regular.

The $*$-integral transform construction provides an equivalence
  $$
  \xymatrix{
  \Phi : \DCoh_{prop/Y}(X \times_S Y) \ar[r]^-\sim &   \Fun^{ex}_{\Perf S}(\Perf X , \Perf Y)
}  $$
\end{corollary}

Before giving the proof of the theorem, we make some remarks and discuss an alternate formulation.
\begin{remark} It is possible to relax the assumptions on $p_X$ a little,
  \footnote{For instance, it should be possible to prove the Theorem for $p_X$ a relative tame DM stack. The key extra input, due to Abramovich-Olsson-Vistoli, is the following: If $q\colon X \to X'$ is the coarse moduli space, then $q_*$ is $t$-exact and \'etale locally on $X'$ a global quotient by a finite flat linearly reductive group scheme.  One can use this to reduce \autoref{prop: push dcoh} for $X$ to it for $X'$, analogous to how \autoref{lem: aff aperf} is used elsewhere.}
  but it seems difficult to make a general statement if $X$ is allowed to be genuinely stacky.  To see why, consider already the case of $S = Y = \pt$ and $X = B\GG_m$.  Then, 
  \[ \Perf X  \isom \bigoplus_{\ZZ} \Perf k \qquad \text{so that} \Fun^{ex}(\Perf X, \Perf k) = \prod_{\ZZ} \Perf k \]
That is, a functor $F$ is determined by the (arbitrary) collection of complexes $F(\O(n))$ where $\O(n)$ is line bundle corresponding to the degree $n$ character on $\GG_m$.  Note that
\[ F = \Phi_{\K} \qquad \text{for} \quad \K = \bigoplus_n \O(-n) \otimes F(\O(n)) \]
Unfortunately, there does not seem to be an existing ``geometric'' name for the finiteness condition enjoyed by this $\K$.  In particular, $\K$ need not be $t$-bounded either above or below in general.
\end{remark}

%
%

\begin{remark}
  Suppose that $S$ admits a dualizing complex and that $Y \to S$ is also locally of finite presentation.  Then, $X, Y, X \times_S Y$ also admit dualizing complexes given by $!$-pullback. In this situation we may alternatively let  
  \[ \omega_X \otimes \Perf X = \left\{ \omega_X \otimes P \in \QCsh(X) : P \in \Perf X \right\} \subset \QCsh(X) \]
  Equivalently, this is the essential image of the Grothendieck duality functor on $\Perf X$.

  Then, one can show that the $!$-integral transform gives an equivalence
  \[ \Phish\colon \DCoh_{\proper/Y}(X \times_S Y) \stackrel\sim\longrightarrow    \Fun^{ex}_{\Perf Y}(\omega_X \otimes \Perf X , \DCoh Y) \]

  More precisely, the following Lemma gives rise to a commutative diagram
  \[\xymatrix{
  \DCoh_{\proper/Y}(X \times_S Y) \ar[d]_{\sim}^{\DD} \ar[r]^-{\Phi} & \ar[d]_{\sim}^{\DD \circ - \circ \DD}  \Fun^{ex}_{\Perf Y}(\Perf X , \DCoh Y) \\
  \DCoh_{\proper/Y}(X \times_S Y)^{op} \ar[r]_-{\Phish} &   \Fun^{ex}_{\Perf Y}(\omega_X \otimes \Perf X , \DCoh Y)^{op}
  } \]
  where the vertical arrows are equivalences by Grothendieck duality and the top arrow is the equivalence of \autoref{thm: dcoh shriek}.
\end{remark}

\begin{lemma}\label{lem:groth-dual-functors}  \begin{enumerate}
    \item The Grothendieck duality functor restricts to an equivalence
      \[ \DD \colon \Perf(X)^{op} \stackrel\sim\longrightarrow \omega_X \otimes \Perf X \]
    \item Suppose that $\F \in \DCohp(X)$ and $\K \in \DCohp(X \times_S Y)$. Then, there is a natural equivalence
     \[ (p_1)^! (\DD \F) \shotimes \DD(\K)\isom \DD\left((p_1)^* \F \otimes \K\right)  \]
     \item Suppose that $\F$, $\K$ as in (ii) and that each homology sheaf of $(p_1)^* \F \otimes \K$ has support proper over $Y$.  Then, there is a natural equivalence
       \[ \Phish_{\DD\K}(\DD\F)\isom \DD\Phi_{\K}(\F) \]
  \end{enumerate}
\end{lemma}
\begin{proof} 
 For (i): Note that $\DD(\P) = \P^\dual \otimes \omega_X$ by dualizability of $\P$, so that it suffices to prove that $\DD$ is fully-faithful on $\Perf X$.  We wish to show that the natural map 
  \[ \RHom_{X}(\P, \P') \to \RHom_{X}(\DD\P, \DD\P') \] is an equivalence.
  The claim is smooth local on $X$ (taking care that the restriction of a dualizing complex along a smooth map is again a dualizing complex!), so we may suppose that $X$ is affine.  The subcategory of $\Perf X^{op} \times \Perf X$ consisting of those pairs $(\P, \P')$ for which this is an equivalence is closed under finite limits, finite colimits, and retracts in each variable.  Thus, it suffices to show that it contains $(\O_X, \O_X)$.  This is part of the definition of a dualizing complex (note that the restriction of a dualizing complex to.

  For (ii), it is either a chase of well-definedness or an immediate consequence of the definition of $!$-pullback and Grothendieck duality for $\DCohpm$ (\autoref{prop:groth-duality}).  For (iii), it follows from (ii) and Grothendieck-Serre duality (i.e., the compatibility of properly supported pushforward with duality).
\end{proof}

The proof of \autoref{thm: dcoh shriek} occupies the rest of this section.  We will begin with several preliminary results.  The following Lemma tells us that we looking for a subcategory of $\QC(X \times_S Y)$:

\begin{lemma} Suppose that $i \colon \C \subset \QC(Y)$ is the inclusion of a full subcategory closed under finite colimits, retracts, and tensoring by objects of $\Perf S$.  Then,  \label{lem:fully-faithful} \begin{enumerate}
    \item  The natural functor
      \[ i_* \colon \Fun^{ex}_{\Perf S}(\Perf X, \C) \longrightarrow \Fun^{ex}_{\Perf S}(\Perf X, \QC(Y)) \]
      is fully faithful.
    \item The star-integral transform restricts to an equivalence
      \[ \left\{ \K \in \QC(X \times_S Y) : \Phi_{\K}(\Perf X) \subset \C \right\} \stackrel{\sim}\longrightarrow \Fun^{ex}_{\Perf S}(\Perf X, \C) \]
\item The $*$-integral transform construction restricts to a fully faithful functor
  $$
  \xymatrix{
  \Phi \colon \DCoh_{\proper/Y}(X \times_S Y) \ar[r] &   \Fun^{ex}_{\Perf S}(\Perf X , \DCoh Y)
}  \qedhere $$
\end{enumerate}
\end{lemma}

\begin{proof} \begin{enumerate}
  \item This is a general fact about exact functors between stable idempotent complete categories.   That is, we reduce to the assertion that if $\D, \E, \E'$ are stable, idempotent complete, and $i \colon \E \to \E'$ exact, then $i_* \colon \Fun^{ex}(\D, \E) \to \Fun^{ex}(\D, \E')$ is fully faithful.  Ignoring set-theoretic issues, we prove this as follows: By taking $\Ind$ and identifying right exact functors with compact objects and colimit preserving functors, we reduce to showing that $\Fun^L(\Ind \D, -)$ preserves that $\Ind(i)$ is a monomorphism.  This follows by noting that $\Ind(i)$ admits a \emph{colimit preserving} right adjoint exhibiting it as fully faithful, and this is preserved by $\Fun^L(\Ind \D, -)$.
    \item  Recall that $X \to S$ is a perfect morphism by \autoref{lem:very good}, so that \cite{BFN}*{Theorem~4.14} implies that
  \[ \Phi \colon \QC(X \times_S Y) \stackrel{\sim}\longrightarrow \Fun^L_{\QC S}(\QC X, \QC Y) = \Fun^{ex}_{\Perf S}(\Perf X, \QC Y) \]
      so that (ii) follows from (i) and the previous displayed equation.
    \item It suffices to show that given $\K \in \DCoh_{prop/Y}(X \times_S Y)$ and 
  $\P\in \Perf X$ that
$
\Phi_\K(\P) = p_{2*}(p_1^*(\P) \otimes \K)
$
is coherent.
Observe that $p_1^* (\P) \otimes \K $ is coherent (since $p_1^* (\P)$ is perfect and $\K$ is coherent) with support proper over $Y$, and properly supported pushforward preserves coherence under our hypothesis of finite cohomological dimension.\qedhere
\end{enumerate}
\end{proof}

It remains to identify this subcategory.  As a first step, we begin with the following strong generation result in the spirit of \cite{BvdB}*{Theorem~3.1.1}.  Like op.cit. it is based on the extension result of Thomason-Trobaugh \cite{TT}, adjusted to algebraic spaces by replacing the use of Mayer-Vietoris squares with Nisnevich-type ``excision squares'' as in \cite{DAG-XII}.
\begin{lemma}\label{lem:very good} Suppose that $\pi \colon X \to S$ is a quasi-compact and quasi-separated relative algebraic space with $S$ quasi-compact.  Then, $\pi$ is of finite cohomological dimension.  Furthermore, if $S = \Spec A$ is affine then there exists a single perfect $G$ that generates $\QC(X)$.
\end{lemma}
\begin{proof} Since $S$ is assumed quasi-compact, we can reduce to the case of $S$ affine.  The assertion on cohomological dimension, and the fact that $\QC(X)$ is generated by $\Perf(X)$, now follows from \cite{DAG-XII}*{Corollary~1.3.10, Corollary~1.5.12}.  For the last assertion we must obseve that the proof of Theorem~1.5.10 in op.cit. can be modified to show that a \emph{single} perfect object $G$ generates: This is certainly true on affines, and one uses the Thomason-Trobaugh lifting argument to show that this property glues under excision squares as follows.  More precisely, suppose that 
  \[ \xymatrix{
  U' = \Spec R \times_X U \ar[r]^--{j'} \ar[d]^{\eta'} & \Spec R \ar[d]^{\eta} \\ 
U \ar[r]_j & X \\
  } \]
  is an excision square as in op.cit.  Note that in this case there is a pullback diagram of categories
  \[ \QC(X) \stackrel\sim\longrightarrow \QC(U) \times_{\QC(U')} \QC(\Spec R). \]

  \medskip

  {\noindent}{\bf Two preliminary constructions:}\\
  A first preliminary construction: Note that $j'$ is a quasi-compact open immersion so that we may pick $f_1, \ldots, f_r \in H_0 R$ that cut out the closed complement $Z'$; let $K' \in \Perf_{Z'}(\Spec R)$ be a single perfect complex which generates $\QC_{Z'}(\Spec R)$ under shifts and colimits, for instance the Koszul complex of the $f_i$.  Finally, since $(j')^* K' = 0$ we may, using the above pullback diagram, uniquely lift $K'$ to an object $K \in \Perf X$ satisfying $\eta^* K = K'$ and $j^* K = 0$.

  A second preliminary construction: We may suppose that there exists $G_U \in \Perf U$ that generates $\QC U$ under shifts and colimits.  We claim that there exists some $G \in \Perf X$ such that $j^* G \isom G_U \oplus G_U[+1]$: By the previously displayed pullback diagram, we are reduced to lifting $\eta'^* G_U \oplus \eta'^* G_U[+1]$ to an object in $\Perf(\Spec R)$.  This can be done by the Thomason-Trobaugh lifting trick  (c.f., \cite{DAG-XI}*{Lemma~6.19}).

  \medskip

  {\noindent}{\bf The generator:}\\
  We now claim that $G \oplus K$ generates $\QC(X)$ under shifts and colimits.  It is enough to show that $\RHom_X(G \oplus K, \F) = 0$ implies that $\F = 0$.  Suppose that $\RHom_X(G \oplus K, \F) = 0$.  Note first that by construction of $K$ we have
  \begin{equation}\label{eqn:closed}
    \RHom_X(K, \F) = \RHom_U(j^* K, j^* \F) \times_{\RHom_{U'}(\eta'^* j^* K, \eta'^* j^* \F)} \RHom_{\Spec A}(\eta^* K, \eta^* \F) = \RHom_{\Spec A}(K', \eta^* \F) \end{equation}
  and that
  \begin{equation}\label{eqn:open} \RHom_X(G, j_* j^* \F) = \RHom_U(j^* G, j^* \F) \text{ contains } \RHom_U(G_U, j^* \F) \text{ as a retract}. \end{equation}  We will use the first of these to show that our assumption implies that $\F = j_* j^* \F$, and the second to conclude by showing that $j_* j^* \F = 0$.

    \medskip

    {\noindent}{\it Reducing to $\F = j_* j^* \F$:}\\
  Consider the fiber sequence
  \[ \F_Z \longrightarrow \F \to j_* j^* \F \]
  where $\F_Z$ is supported on the closed complement $Z = X \setminus U$.  We will first show that $\F_Z = 0$.  Since $j^* \F_Z = 0$ it is enough to show that $\eta^* \F_Z = 0$. Note that $\F_Z$ is supported on $Z$, so that $\eta^* \F_Z$ is supported on $Z'$.  Applying \autoref{eqn:closed} we conclude that
  \[ 0 = \RHom_X(K, \F) = \RHom_X(K, \F_Z) = \RHom_{\Spec A}(K', \eta^* \F_Z) \]
  so that by the choice of $K'$ we have $\eta^* \F_Z = 0$ and hence $\F_Z = 0$.

    \medskip
    {\noindent}{\it Showing $j^* \F = 0$:}\\
  Consequently $\F = j_* j^* \F$.  By our hypothesis and \autoref{eqn:open} we conclude that
\[ 0 = \RHom_X(G, \F) = \RHom_X(G, j_* j^* \F) = \RHom_U(j^* G, j^* \F) \]
and since $G_U$ is a retract of $j^* G$ this implies $\RHom_U(G_U, j^* \F) = 0$.  By assumption on $G_U$, this implies $j^* \F = 0$.  We conclude that $j^* \F = 0$ and so $\F = j_* j^* \F = 0$.  This completes the proof.
\end{proof}

%
%

Finally we embark on identifying the subcategory of interest. (The space $Z$ below will play the role of $X\times_S Y$ in the argument.)
\begin{prop}\label{prop: push aperf} Suppose that $p\colon Z \to S = \Spec A$ is a Noetherian separated $S$-algebraic space over a Noetherian affine base $S$.  Then the following conditions on $\F \in \QC(Z)$ are equivalent:
\begin{enumerate}
      \item $\F \in \DCohp(Z)$ and each homology sheaf $H_i(\F)$ has support proper over $S$;
      \item $\F \in \QC(Z)_{>-N}$ for some $N$, and $\RGamma(Z, \H \otimes \F) \in \DCohp(A)$ for all $\H \in \DCoh(Z)$;
      \item $\RGamma(Z, \H \otimes \F) \in \DCohp(A)$ for all $\H \in \DCohp(Z)$.
\end{enumerate}
\end{prop}
\begin{proof}
  Note that (i) implies (ii) and (iii) by the proper pushforward theorem.  Furthermore (ii) implies (iii) by an approximation argument using that $\F$ is bounded below and that $Z$ has bounded cohomological dimension (see the proof of \autoref{prop: push dcoh} for a similar argument).  Conversely, (iii) implies (ii) by applying \autoref{lem:fake-lazard} to the functor $\RGamma(Z, \G^* \otimes -)$ where $\G^*$ is a generator guaranteed to exist by \autoref{lem:very good}.  It remains to show that (ii) and (iii) together imply (i), and we do this through a series of reductions.\footnote{Note that (ii) is only used in the first reduction.  If we could eliminate it, this proof would be cleaner!}

  \medskip

  {\noindent}{\it Reduction to $Z$ and $A$ classical (i.e., underived):} By (ii) we have that $\F$ is bounded below, so it suffices to prove that it has coherent homology sheaves with support proper over $S$.  Since we have proven that (i) implies (ii) and (iii), it suffices to prove that the lowest degree homology sheaf has this property:  applying this iteratively to the fibers of the natural maps will prove the result.  
  
  Without loss of generality suppose that $\F \in \QC(Z)_{\geq 0}$ and we will show that $H_0(\F)$ is coherent and properly supported over $S$.  Suppose now that $i \colon Z_{cl} \to Z$ is the inclusion of the underlying classical algebraic space.  Then, 
  \[ i_* i^* \F \to \F \] 
  induces an isomorphism on $H_0(\F)$.  Note that $i_*$ is $t$-exact and induces an equivalence on hearts, and that the proper support condition is topological, so that it suffices to show that $H_0(i^* \F)$ is coherent on $Z_{cl}$ with proper support over $S_{cl}$.  Note that
  \[ \RGamma(Z_{cl}, \H \otimes i^* \F) = \RGamma(Z, i_*(\H) \otimes \F) \in \DCohp(A) \] since $i_*(\H) \in \DCohp(\H)$, and consequently 
  \[ \RGamma(Z_{cl}, \H \otimes i^* \F) \in \DCohp(H_0 A) \] since the inclusion detects the property of being bounded below with coherent homologies.
  
  Thus, we may assume that $Z$ and $S$ are classical.
  
  \medskip

  {\noindent}{\it Reduction to $p$ finite-type proper: } Let $i \colon Z \to Z'$ be affine with $Z'$ a finite-type and separated algebraic space over $A$: such a morphism exists by relative ``Noetherian'' approximation for separated algebraic spaces.  Let $\ol{p} \colon \ol{Z} \to S$ be a proper morphism and $j \colon Z' \to \ol{Z}$ an open immersion; such data exists by the Nagata compactification theorem for algebraic spaces \cite{CLO-Nagata}.  Blowing up $\ol{Z}$ along the reduced induced structure on $\ol{Z} \setminus Z'$, we may suppose that $j$ is affine.  Since $j$ was quasi-compact, the blowup remains of finite-type.  Consequently $j \circ i \colon Z \to \ol{Z}$ is an affine map to a finite-type and proper algebraic space over $A$.  By \autoref{lem: aff aperf} it suffices to show that $j_* i_* \F \in \DCohp(\ol{Z})$ with propertly supported homology sheaves.  Note that for any $\H' \in \DCohp \ol{Z}$ we have
  \[ \RGamma(\ol{Z}, \H' \otimes (j\circ i)_* \F) = \RGamma(\ol{Z}, (j \circ i)_*((j\circ i)^*\H' \otimes \F)) = \RGamma(Z, (j\circ i)^*\H' \otimes \F) \in \DCohp(A) \] since $i^* j^*\H' \in \DCohp(\ol{Z})$.   Thus, we may assume that $Z$ is proper over $S$ and it suffices to show that $\F \in \DCohp(Z)$.

  \medskip

  {\noindent}{\it Reduction to $p$ finite-type projective: } By Chow's Lemma for algebraic spaces \cite{Knutson-AlgSp}*{IV.3.1} there exists $q\colon \sq{Z} \to Z$ with $\sq{Z}$ projective and birational over $Z$.  Thus, 
  \[ \cone(q_* q^* \F \to \F) \]
  is supported on a proper closed subset of $Z$.  By Noetherian induction, it suffices to show that $q_* q^* \F \in \DCohp(Z)$.  By the proper pushforward theorem, it suffices to show that $q^* \F \in \DCohp(\sq{Z})$.  Note that for any $\H' \in \DCohp(\sq{Z})$
  \[ \RGamma(\sq{Z}, \H' \otimes q^* \F) = \RGamma(Z, q_*(\H' \otimes q^* \F)) = \RGamma(Z, q_*(\H') \otimes \F) \in \DCohp(A) \] since $q_*(\H') \in \DCohp(Z)$ by the proper pushforward theorem.

  Thus, we may assume that $Z$ is finite-type projective over $S$.

  \medskip

  {\noindent}{\it Reduction to $p$ a projective space: } Since $Z$ is finite-type projective, and the claim is local on $S$, we may suppose that there is a closed immersion $i \colon Z \to \PP_S^n$.  This morphism is affine, so as before, we reduce to showing that $i_* \F \in \DCohp(\PP_S^n)$.

  \medskip

  {\noindent}{\it Proof in the case of $Z = \PP_S^n$: } Let $\Delta\colon Z \to Z\times_S Z$ be the relative diagonal.  Using Beilinson's method of resolution of the diagonal, the diagonal $\Delta_* \O_Z$ can be built in finitely many steps by taking cones, shifts, and retracts from objects of the form $\P \boxtimes_S \P'$, with $\P, \P' \in \Perf Z$. 

Given $\F\in \QC(Z)$, we can express it in the form
$$
\F \simeq (p_2)_*\left(\Delta_* \O_Z \otimes (p_1)^* \F\right)
$$
And consequently it may be built in  finitely steps by taking cones, shifts, and retracts from objects of the form 
$$ 
\xymatrix{
 (p_2)_*\left((\P \boxtimes_S \P') \otimes (p_1)^* \F\right)  \simeq  \P' \otimes_A \RGamma(Z,\P \otimes \F)  &
 \P, \P' \in \Perf Z
 }
 $$
If  $\F$ satisfies  condition (ii), then we have $\RGamma(Z, \P \otimes \F) \in \DCohp(A)$, hence $\P' \otimes_A \RGamma(Z, \P \otimes \F) \in \DCohp(Z)$.  Consequently, $\F \in \DCohp(Z)$ as desired.  Note that since $\pi$ is projective, the support condition is vacuous.
  \end{proof}

\begin{lemma} \label{lem: aff aperf}
Suppose $p\colon Z \to S$ is an affine map of Noetherian algebraic spaces.  Then the following conditions on $\F \in \QC(Z)$ are equivalent:
  \begin{enumerate}
      \item $\F \in \DCohp(Z)$ and $H_i(\F)$ has support finite over $S$ for all $i$;
      \item $\F \in \DCohp(Z)$ and $H_i(\F)$ has support proper over $S$ for all $i$;
      \item $p_* \F \in \DCohp(S)$.
    \end{enumerate}
  \end{lemma}
  \begin{proof}
  Let $\supp_Z(\F)$ denote the support of $\F$.
    Observe (i) $\iff$ (ii) since $\supp_Z(H_i(\F)) \to Z$ is affine, $p$ is affine, and proper + affine = finite.
    
    Note that (i) $\implies$ (iii) by the proper pushforward theorem (in the easy finite case).

    To show that (iii) $\implies$ (i), without loss of generality, we may assume that $S$, and hence $Z$ as well, is affine.  Since $p_*$ is conservative and $t$-exact, it is clear that $H_i \F = 0$ for $i \ll 0$, since this is so for $p_* H_i \F$. Hence it remains to show that each $H_i \F$ is coherent over $H_0 \O_X$ with support finite over $H_0 \O_S$.  
     We are thus reduced to the proving the following statement in commutative algebra: Suppose given a map of classical rings $\phi\colon A \to B$ and a $B$-module $A$ such that $B$ is coherent as an $A$-module, then $M$ is coherent as $B$-module and $B/\mathrm{Ann}_B(M)$ is finite over $\Spec A$.  
    
    To prove this statement, (using our standing Noetherian hypotheses) we may replace ``coherent'' with ``finitely-generated'' in the assertion.  But now, if $m_1, \ldots, m_n$ generate $M$ as an $A$-module, then they also generate $M$ as a $B$-module since $\sum a_i m_i = \sum \phi(a_i) m_i$ by definition.  Furthermore, in this case $\mathrm{Ann}_B(M) = \cap_i \mathrm{Ann}_B(m_i)$ so that it suffices to show that $B/\mathrm{Ann}_B(m_i)$ is finite over $A$ for each $i$.  This is isomorphic to the sub-$B$-module of $M$ generated by $m_i$, and is again finite  as a submodule of the finite $A$-module $M$ (using again our Noetherian hypotheses).
  \end{proof}

We can use the above to deduce:
\begin{prop} \label{prop: push dcoh} Suppose $p\colon Z \to S = \Spec A$ is a  Noetherian separated $S$-algebraic space over a Noetherian affine base $S$.  Then the following conditions on $\F \in \QC(Z)$ are equivalent:
  \begin{enumerate}
      \item $\F \in \DCoh(Z)$ with support proper over $S$;
      \item $\F$ is homologically bounded, $\F\in \DCohp(Z)$, and $H_i(\F)$ has support proper over $S$ for all $i$;
      \item $\F$ is homologically bounded, and $\RGamma(Z, \H \otimes \F) \in \DCohp(Z)$ for all $\H \in \DCohp(Z)$;
      \item $\RGamma(Z, \P \otimes \F) \in \DCoh(Z)$ for all $\P \in \Perf(Z)$.
    \end{enumerate}
\end{prop}
\begin{proof} 
  Note that (i) is equivalent to (ii) by definition, and (ii) is equivalent to (iii) by \autoref{prop: push aperf}.  Then, (iii) implies (iv) since perfect complexes have finite Tor-amplitude and $\RGamma(Z, -)$ is left $t$-exact.
  
  It suffices to show that (iv) implies (iii).  Suppose that $\F$ satisfies the conditions of (iv).  Note that \autoref{lem:very good} and \autoref{lem:fake-lazard} imply that $\F$ is homologically bounded since $\RGamma(Z, \G^\dual \otimes \F)$ is homologically bounded for $\G$ a single perfect complex generating $\QC(X)$.  Now, without loss of generality we may suppose that $\F$ is connective, i.e., $\F \in \QC(X)_{\geq 0}$.  
  
  Suppose $\H \in \DCohp(Z)$.  We must show that $\RGamma(Z, \H \otimes \F)$ is almost perfect as an $A$-module. Note that $\RGamma(Z, \H \otimes \F)$ is bounded-below by the finite cohomological dimension of $p$, so that it suffices to show that each of its homology modules is coherent over $A$.  Let $d$ be the cohomological dimension of $p$.  Since $\H$ is almost perfect and $Z$ is perfect, there exists a perfect complex $\P$ and a map $\P \to \H$ whose cone is in $\C_{> d-i}$.  Since $\F$ is connective, the cone of the map $\P \otimes \F \to \H \otimes \F$ also lies in $\C_{> d-i}$.  By the boundedness of cohomological dimension, $H^i$ and $H^{i-1}$ of this cone vanish -- so, the natural map
  \[ H^i(Z, \H \otimes \F) \to H^i(Z, \P \otimes \F) \]
  is an isomorphism.  Since the right hand side was coherent by assumption, we are done.
 \end{proof}

\begin{lemma}\label{lem:fake-lazard} Suppose that $\pi \colon X \to S$ is a quasi-compact and separated relative algebraic space, and that $S = \Spec A$ is affine.
  
  If $G$ is a connective perfect complex generating $\QC(X)$,  then
  \[ \RGamma(G^* \otimes -)\colon \QC(X) \to A\mod \]
  is conservative, left $t$-exact up to a shift, and right $t$-exact up to a shift.

  In particular, it detects the properties of being bounded below and bounded above.
\end{lemma}
\begin{proof}  Identifying
  \[ \RGamma(X, G^* \otimes -) = \RHom_X(G, -) \] we see that it is conservative by the assumption that $G$ generates.  If $d$ is the cohomological dimension of $\pi$ and $n$ is such that $G^* \in \QC(X)_{>-n}$ then this functor is right $t$-exact up to a shift by $n+d$.  If $a$ is the Tor amplitude of $G^*$ then this functor is left $t$-exact up to a shift by $a$.

  The result then follows by the following Lemma.
\end{proof}

{\it We warn the reader that the following Lemma is false, requiring a different proof of \autoref{thm: dcoh shriek}. Please see \autoref{prop correction} of \autoref{s:correction} for a corrected statement to stand in its place.}

\begin{lemma} Suppose that $\C, \D$ have all limits and colimits and carry $t$-structures which are both left- and right-complete, and that $F \colon \C \to \D$ is conservative, limit and colimit preserving, left $t$-exact up to a shift, and right $t$-exact up to a shift.  Then, $F$ detects the bounded below / bounded above objects.
\end{lemma}
\begin{proof} Suppose $c \in \C$ is bounded below / bounded above.  The exactness, up to a shift, of $F$ guarantees that $F(c)$ is bounded below / bounded above.

  We now prove the converse implication:

  Let $\ell$ be the error of left $t$-exactness, and $r$ the error of right $t$-exactness.  For any $c \in \C$, there is a natural map
  \[ F(\tau_{>N} c) \longrightarrow \tau_{>N-r} F(c) \]
  whose cofiber is $F(\tau_{\leq N} c) \in k\mod_{\leq N+\ell}$ so that the map is an equivalence on $\tau_{>N+\ell}$.

  Similarly, there is a natural map
  \[ \tau_{\leq N+\ell} F(c) \longrightarrow F(\tau_{\leq N} c) \]
  whose fiber $F(\tau_{>N} c) \in k\mod_{>N-r}$ so that the map is an equivalence on $\tau_{\leq N-r}$.

  The result now follows by noting the standard fact that the truncation above and below functors commute.
\end{proof}

Finally, we are in a position to complete the proof of the theorem:
\begin{proof}[Completing the Proof of \autoref{thm: dcoh shriek}]
To complete the proof of the Theorem, it suffices to see that if $p_{2*}(p_1^*(\F) \otimes \K)\in \DCoh(Y)$ for all $\F\in \Perf(X)$, then in fact   $\K\in \DCoh_{\proper/Y}(X \times_S Y)$.

This is local in $Y$ so we may assume $Y$ is affine and Noetherian, and hence $p_1 \colon X\times_S Y\to X$ is affine. Thus pullbacks $p_1^*(\F)$ generate $\Perf(X \times_S Y)$, and so it suffices to show
that if $p_{2*}(\P \otimes \K)\in \DCoh(Y)$ for all $\P\in \Perf(X)$, then   $\K\in \DCoh_{\proper/Y}(X \times_S Y)$.
This will follow from applying \autoref{prop: push dcoh} above with $Z=X\times_S Y$, $S=Y$, and $p=p_2$.
\end{proof}

%
%

\section{Shriek integral transforms revisited}\label{sect: shriek}

\begin{na} For quasi-coherent complexes, it is often the case that categories of functors coincide with quasi-coherent complexes on the fiber product.  For ind-coherent complexes, this is often true over a point but very rarely true over a non-trivial base.  The goal of the present section is to ``fix'' this, or rather to show that the failure of the shriek integral transform to give an equivalence can be precisely controlled by means of the $t$-structure.  
\end{na}

In order to do this, we will need to consider some constructions involving categories equipped with reasonably behaved $t$-structures: for this we refer the reader to \autoref{app:t-bdd}.  Let us recall the highlights:
\begin{na} In the Appendix, the reader will find a definition of an $\infty$-categorical notion of a (left) \demph{coherent $t$-category}, and a (left) $t$-exact functor between such. The theory of coherent $t$-categories, with left $t$-exact functors, admits a localization ((left) \demph{complete $t$-categories}) and a co-localization ((left) \demph{regular $t$-categories}).  Though we do not explicitly need to use this, the two resulting theories are equivalent by \autoref{thm:reg-coh-equiv}.\footnote{This sort of phenomenon, where a localiation is equivalent to a colocalization, has several more familiar examples.  For instance suppose that $R$ is a commutative ring  and $I \subset H_0 R$ a finitely-generated ideal; then, there is the Greenless-May equivalence \cite{DAG-XII}*{Prop.~4.2.5} between locally $I$-torsion $R$-modules, and $I$-complete $R$-modules.}  The examples of interest to us are:
  \begin{itemize}
    \item For a Noetherian geometric stack, $\QC(X)$ is \demph{complete} by \autoref{prop:qc-t-cplt};
    \item For a geometric stack of finite-type over a char. $0$ field, $\QCsh(X)$ is \demph{regular} by \autoref{prop:qcsh-t-reg}.
  \end{itemize}
  Given $\C, \D$ presentable categories with $t$-structures, one can put a $t$-structure on $\C \otimes \D$ and $\Fun^L(\C, \D)$. The formation of tensor/functors does not generally preserve complete/regular (or often even coherent) categories, giving -- in light of the above -- an obstruction to ``functor theorems.''  
\end{na}

\begin{na} In this section, all stacks will be locally of finite presentation over a perfect field $k$.  This is for three reasons: first, we work finite-type over a reasonable base so that we may use the formalism of Grothendieck duality; second, we want tensoring over the base to be left $t$-exact, so we require it to be a field; finally, we rely on generic smoothness results so we require that $k$ be a perfect.
\end{na}

\begin{na}
Finally, we will have one more \emph{unspoken} assumption: For all $X$ that occur, we will ask that $X$ have finite cohomological dimension and that $\QCsh X = \Ind(\DCoh X)$.  The reason we leave it unspoken is that if $k$ has characteristic zero, then this will be automatic for all the stacks we consider by \cite{DrinfeldGaitsgory}.
\end{na}

The goal of this section is to prove:
\begin{theorem}\label{thm:fun-regulariz} Suppose that $S$ is a quasi-compact, geometric, finitely-presented $k$-stack; that $\pi \colon X \to S$ is a quasi-compact and separated $S$-algebraic space of finite presentation; and, that $Y$ is an $S$-stack almost of finite presentation.  Then the $!$-integral transform 
  \[ \Phish\colon \QCsh(X \times_S Y) \longrightarrow \Fun^L_{\QCsh S}(\QCsh X, \QCsh Y) \]
  has the following properties:
  \begin{enumerate}
    \item There exists an integer $N$ such that the image of $\QCsh(X \times_S Y)_{<N}$ lies in the left $t$-exact functors.
    \item It is fully-faithful on the full subcategory $\QCsh(X \times_S Y)_{<0}$.
    \item There exists an integer $M$ such that any left $t$-exact functor lies in the essential image of $\QCsh(X \times_S Y)_{<M}$.
  \end{enumerate}
\end{theorem}

From this, we will deduce:
\begin{corollary}\label{cor:dcoh-t-reg} Suppose $S, X, Y$ are as in \autoref{thm:fun-regulariz} and furthermore that $S$ is regular.  Then, there is a $t$-structure on the functor category with
  \[ \Fun^L_{\QCsh S}(\QCsh X, \QCsh Y)_{\leq 0} \eqdef \left\{ F \in \Fun^L_{\QCsh S}(\QCsh X, \QCsh Y) : \text{$F$ is left $t$-exact} \right\} \]

With this $t$-structure, and the ordinary $t$-structure on $\QCsh(X \times_S Y)$, the $!$-integral transform is left $t$-exact up to a shift and exhibits $\QCsh(X \times_S Y)$ as the regularization of the functor category.
\end{corollary}

Before giving a proof of the Theorem and its Corollary, let us note two more concrete consequence:
\begin{corollary}\label{cor:shriek-bounded} Under the hypotheses of \autoref{thm:fun-regulariz}, the $!$-integral transform induces an equivalence
  \[ \QCsh(X \times_S Y)_{<\infty} \stackrel{\Phish}\longrightarrow \Fun^{L, t<\infty}_{\QCsh(S)}(\QCsh(X), \QCsh(Y)) \]
  where the superscript ``$t<\infty$'' denotes the full subcategory of functors which are left $t$-exact up to a finite shift.\end{corollary}
\begin{proof} It is well-defined by (i), fully-faithful by (ii), and essentially surjective by (iii).
\end{proof}

\begin{corollary}\label{cor:compare-phi-phish} Suppose furthermore that $S$ is regular and that $X$ is of finite Tor dimension over $S$.  In this case, $\omega_X \otimes - $ takes $\Perf X$ into $\DCoh X$.  Let $R$ denote the functor $\QCsh(-) \to \QC(-)$.   Then, one has a commutative diagram
  \[ \xymatrix{
  \QCsh(X \times_S Y)_{<\infty} \ar@{^{(}->}[d]_R \ar_-{\sim}[r]^-{\Phish} &   \Fun^{ex, t<\infty}_{\Perf(S)}(\DCoh X, \QCsh Y) \ar[d]^r \\
  \QC(X \times_S Y) \ar[r]^-{\sim}_-{\Phi} &  \Fun^{ex}_{\Perf(S)}(\Perf X, \QC Y) 
  } \]
  where \[ r(F) = R \circ F \circ (\omega_X \otimes -) \]  
  
  In particular, $r$ is fully faithful.  The essential image of $r$ consists of those functors whose Kan extension to functors $ F \in \Fun^L_{\QC S}(\QC X, \QC Y)$ satisfy the following condition: There exists a constant $N$ such that if $\G \in \QC X$ has Tor amplitude at most $a$, then $F(\G) \in \QC Y_{<N+a}$.   
\end{corollary}
\begin{proof} The commutativity of the diagram is provided by the module structure for $\QCsh$ over $\QC$:
  \begin{align*} R \circ \Phish_{\K}(\omega_X \otimes \P) &= R \circ p_{2_*}\left( p_1^!(\omega_X \otimes \P) \shotimes \K\right) \\
    &= p_{2_*}\left( R \left((\omega_{X \times_S Y} \otimes p_1^*(\P)) \shotimes \K\right) \right) \\
    &= p_{2_*}\left( R \left( \omega_{X \times_S Y} \shotimes (p_1^*(\P) \otimes \K) \right)\right) \\
    &= p_{2_*}\left( R \left( p_1^*(\P) \otimes \K \right) \right)  \\
    &= p_{2_*}\left( p_1^*(\P) \otimes R(\K) \right) = \Phi_{R(\K)}(\P) 
  \end{align*}
  where $\K \in \QCsh(X \times_S Y)_{<\infty}$ and $\P \in \Perf X$.

  The  equivalence and fully-faithfulness annotations follow from the Theorem, the functor theorem for $\QC$, and the fact that $R$ induces an equivalence $\QCsh(-)_{<\infty} \isom \QC(-)_{<\infty}$.

It remains to identify the essential image:  By the previous diagram, we must show that for $\K \in \QC(X \times_S Y)$ we have $\K \in \QC(X \times_S Y)_{<\infty}$ iff $\Phi_\K$ has the given property.  We may suppose that $S$ is affine, by considering tensoring by the (flat) object $\O_U \in \QC(S)$ associated to a flat affine cover $U \to S$.  Then, we may suppose that $\Perf X$ admits a single generator $\G$, so that $p_{2*} \RHom_{X \times_S Y}(p_1^* \G, -) \colon \QC(X \times_S Y) \to \QC(Y)$ satisfies the hypotheses of \autoref{lem:fake-lazard}.  Since $\G$ is perfect, $\G^*$ has bounded Tor amplitude so that $\Phi_{\K}(\G^*)$ is left bounded; this completes the proof.
\end{proof}

\begin{remark}
In the description above, one could weaken the condition to $\G \in \Perf X$ of Tor amplitude at most $a$ -- and restrict the domain of definition of $F$ to $\Perf X$ -- provided that a weak form of Lazard's Theorem held for $X$: One requires that any $\G \in \QC X$ which is flat is a filtered colimit of perfect complexes of uniformly bounded Tor amplitude.  We have not considered when this holds.
\end{remark}


The strategy will be three-fold: we prove an (absolute, i.e., $S = \Spec k$) ``tensor'' rest; we prove an (absolute) ``duality'' result; and then we pass from the absolute case to the relative case.

We start by recording the following ``absolute'' tensor statement:
\begin{prop}\label{prop:absolute-fun} Suppose that $S = \Spec k$ for a perfect field $k$, that $X \to S$ is a quasi-compact and separated $S$-algebraic space of finite presentation, and that $Y \to S$ is an arbitrary $S$-stack of finite presentation.  Then,
  \begin{enumerate}
      \item The exterior product
        \[ \QCsh(X) \otimes_{\QC S} \QCsh(Y) \longrightarrow \QCsh(X \times_S Y) \]
        is well-defined and is an equivalence.
      \item The $!$-integral transform gives an equivalence
        \[ \Phish \colon \QCsh(X \times_S Y) \longrightarrow \Fun^L_{\QC(S)}(\QCsh(X), \QCsh(Y)). \]
  \end{enumerate}
\end{prop}
\begin{proof} See \cite{toly}*{Proposition B.1.1, Theorem~B.2.4} (or \cite{indcoh}), for the case of $k$ characteristic zero and slightly different hypotheses on $X, Y$.  The same argument in fact works under the hypotheses given here, with \autoref{lem:very good2} providing the necessary dualizability results for $\QCsh(X)$ over $k$.
\end{proof}

Next we consider the (absolute) dualizability statement.  Our input is the following key boundedness property of Grothendieck duality:
\begin{prop}\label{prop:dcoh-t-dual} Suppose $X$ that is a Noetherian geometric stack admitting a dualizing complex $\omega_X$ and of finite cohomological dimension.  Let $\DD$ denote the duality functor with respect to $X$. Then:
      
  The Grothendieck duality functor
        \[ \DD \colon \DCoh X^{op} \stackrel{\sim}\longrightarrow \DCoh X \]
        is an equivalence, and it is left and right $t$-exact up to a finite shift.
\end{prop}
\begin{proof} See \autoref{prop:groth-duality}.
\end{proof}

We re-interpret it as a dualizability result for $\QCsh(X) \isom \Ind(\DCoh(X))$:
\begin{corollary}\label{cor:dcoh-t-dual} Suppose $X$ is of finite cohomological dimension, admits a dualizing complex, and is such that $\QCsh(X) = \Ind(\DCoh(X))$.  Then, one can equip $\QCsh(X)$ with the following alternate $t$-structure henceforth denoted $\ol{\QCsh}(X)$:
      \[ \ol{\QCsh}(X)_{>0} = \Ind\left( \DD(\DCoh X_{<0}) \right)   \qquad \ol{\QCsh}(X)_{\leq 0} = \Ind\left( \DD(\DCoh X_{\geq 0}) \right) \]
      
    Then:
    \begin{enumerate}
      \item $\ol{\QCsh}(X)_{\leq 0}$ consists precisely of those $\F \in \QCsh(X)$ such that 
        \[ \RGamma(X, \F \shotimes -)\colon \QCsh(X) \to k\mod \] is a left $t$-exact functor;
      \item The $!$-integral transform gives a $t$-exact equivalence
        \[ \ol{\QCsh}(X) \longrightarrow \Fun^{L}(\QCsh X, k\mod) \]
      \item Suppose that $\C$ is a presentable $\infty$-category with $t$-structure compatible with filtered colimits.  Then, the equivalence
        \[ \ol{\QCsh}(X) \otimes \C \stackrel{\sim}\longrightarrow \Fun^L(\QCsh(X), \C) \]
        \[ \F \otimes c \mapsto \RGamma(X, \F \shotimes -) \otimes c \]
        is a $t$-exact equivalence.
      \item The identity functor $\QCsh X = \ol{\QCsh} X$ is left and right exact up to finite shifts.  In particular, the equivalence of (iii) 
        \[ \QCsh(X) \otimes \C \stackrel{\sim}\longrightarrow \Fun^L(\QCsh(X), \C) \]
        is left and right exact up to finite shifts (with the usual $t$-structure on the left hand-side).
    \end{enumerate}
\end{corollary}
\begin{proof}
  \begin{enumerate}
  \item Let us show that the collection of such $\F$ contains $\DD(\DCoh X_{\geq 0})$ and is closed under filtered colimits.  It is closed under filtered colimits since $\RGamma$ and $\shotimes$ are, and the $t$-structure on $k\mod$ is compatible with filtered colimits.  If $\F' \in \DCoh X_{\geq 0}$, there is an equivalence
  \[ \RGamma(X, \DD\F' \shotimes - ) = \RHom_{\QCsh(X)}(\F', -) \] so that this functor is left $t$-exact since $\F' \in \QCsh(X)_{\geq 0}$.

  \item This is a reformulation of (i).

  \item By construction, $\Ind(\DD)$ gives a $t$-exact equivalence $\ol{\QCsh}(X) \isom \Ind(\DCoh(X)^{op})$.  The functor under considering is now the composite of the $t$-exact equivalence 
    \[ \ol{\QCsh}(X) \otimes \C \stackrel{\Ind\DD}\longrightarrow \Ind(\DCoh(X)^{op}) \]
    and the $t$-exact equivalence (c.f., \autoref{ssec:t-tens})
    \[ \Ind(\DCoh(X))^{op} \otimes \C \isom \Fun^{ex}(\DCoh X, \C). \]
  \item The identity functor is its own adjoint, so being left (resp., right) exact up to a shift implies that it is also right (resp., left) exact up to a shift.   The result without an auxillary $\C$ now follows by \autoref{prop:dcoh-t-dual}(ii).  Tensoring with $\C$ preserves the property of being right $t$-exact up to a shift by construction, completing the proof.\qedhere
\end{enumerate}
\end{proof}

We are now ready to complete the proof:
\begin{proof}[Proof of \autoref{thm:fun-regulariz}]\mbox{}
  \begin{enumerate}
      \item
        Note that $\Phish$ is left $t$-exact up to a shift, since pushforwards, $!$-pullbacks, and exterior products over the ground field are all so.  The only one of these which may be non-obvious is $!$-pullback: Using quasi-compactness the claim is local, so it is enough to consider the case of a smooth morphism where $f^! = \Omega_f^{d}[d] \otimes f^*$ is a shift by $d$ of an exact functor; and the case of a finite morphism where $f^! = \RHom_{\O_X}(f_* \O_Y, -)$ is left $t$-exact.\footnote{Alternatively, one can reduce to the case of an open embedding and a proper morphism, and use the finite cohomological dimension of a proper relative algebraic space.}
\item
  First note that for $S = \pt$, the map is an equivalence by \autoref{prop:absolute-fun}.  Let us reduce the general case to this.

We note that the $!$-integral transform and exterior products produce maps of augmented cosimplicial diagrams
  \[ \xymatrix{
\QCsh(X \times_S Y) \ar[d]^{\Phish} \ar[r] & \QCsh(X \times Y) \ar[d]^{\Phish} \ar[r] & \QCsh(X \times S \times Y) \ar[d]^{\Phish} \ar[r] & \cdots \\
\Fun^L_{\QCsh S}(\QCsh X, \QCsh Y)  \ar[r] \ar[d]^= & \Fun^L(\QCsh X, \QCsh Y)  \ar[r] \ar[d]^{\sim}  & \ar[d]^{\sim} \Fun^L(\QCsh(X \times S), \QCsh Y) \ar[r]   & \cdots \\ 
\Fun^L_{\QCsh S}(\QCsh X, \QCsh Y)  \ar[r] & \Fun^L(\QCsh X, \QCsh Y)  \ar[r]         & \Fun^L(\QCsh X \otimes \QCsh S, \QCsh Y) \ar[r]   & \cdots }
\]
The maps on the bottom row are given by the product structure, and the bottom row is a totalization diagram.

The maps in the top row are given by pushforward along diagonal maps and graphs.  In particular, they are $t$-exact.  Thus, the claim will be proven if we can show that
\[
\QCsh(X \times_S Y)_{<0} \longrightarrow \Tot\left\{ \QCsh(X \times S^{\times \bullet} \times Y)_{<0} \right\}
\]
is fully faithful.  We will in fact show that it is an equivalence.  Note that the natural functor $\QCsh(-) \to \QC(-)$ is $t$-exact, commutes with pushforward, and induces an equivalence $\QCsh(-)_{<0} \to \QC(-)_{<0}$.  Thus the result follows by examining the diagram
  \[ \xymatrix{
\QCsh(X \times_S Y) \ar[r] \ar[d] & \QCsh(X \times Y) \ar[d] \ar[r] & \QCsh(X \times S \times Y) \ar[d] \ar[r] & \cdots \\
\QC(X \times_S Y)  \ar[r] & \QC(X \times Y) \ar[r] & \QC(X \times S \times Y) \ar[r] & \cdots }
\]
where the bottom row is a $t$-exact totalization diagram by the tensor product theorem for $\QC$.
\item We first handle the case of $S = \pt$.  In this case, 
  \[ \Fun^L(\QCsh X, \QCsh Y) = \Fun^{ex}(\DCoh X, \QCsh Y) = \Ind(\DCoh X^{op}) \otimes \QCsh Y \] carries a $t$-structure, whose co-connective objects are precisely the left $t$-exact functors.  The result then follows from \autoref{cor:dcoh-t-dual}.  Let $N$ be the constant for $X$ and $Y$ for $S = \pt$.

  We now reduce the general case to this: Suppose that 
  \[ F \in \Fun^L_{\QCsh(S)}(\QCsh(X), \QCsh(Y)) \mapsto F_\bullet \in \Tot\left\{ \Fun^L(\QCsh(X) \times \QCsh(S)^{\otimes \bullet}, \QCsh(Y)) \right\} \] 
  is such that $F_0\colon \QCsh(X) \to \QCsh(Y)$ is left $t$-exact.  It follows that each $F_\bullet$ is left $t$-exact up to a finite shift as well, since it is equivalent to the composite
  \[ \QCsh(X) \otimes \QCsh(S)^{\otimes \bullet} \stackrel{\boxtimes^{-1}}\longrightarrow \QCsh(X \times S^{\bullet}) \longrightarrow \QCsh(X) \stackrel{F_0}\longrightarrow \QCsh(Y) \]
  of the inverse of the exterior product equivalence (which is left $t$-exact, since the exterior product over $k$ is right $t$-exact), the pushforward map (which is left $t$-exact), and $F_0$.

  By (ii) and (iii) in the case of a point, applied to $F_\bullet$, there exists an essentially unique
  \[ \F_\bullet \in \Tot \left\{ \QCsh(X \times S^{\times \bullet} \times Y)_{<\infty} \right\} \]
  such that $\Phish(\F_\bullet) \isom F_\bullet$.  Since $F_0$ is left exact, we have that $\F_0 \in \QCsh(X \times Y)_{<N}$.  Since pushforward is left $t$-exact, it follows that $\F_\bullet \in \QCsh(X \times S^{\times \bullet} \times Y)_{<N}$ for all $\bullet$.  It follows from the proof of (ii) that $\F_\bullet$ is the image of some $\F \in \QCsh(X \times_S Y)_{<N}$
  \qedhere
\end{enumerate}
\end{proof}

\begin{remark} Let 
  \[ \theta \colon \Fun^L_{\QCsh(S)}(\QCsh(X), \QCsh(Y)) \longrightarrow \Fun^L(\QCsh(X), \QCsh(Y)) = \Fun^{ex}(\DCoh X, \QCsh(Y)) \] be the natural conservative functor of forgetting $\QCsh(S)$-linearity.  We have seen that the right hand side carries a $t$-structure with $(-)_{\leq 0}$ consisting of the left $t$-exact functors.  Let the superscript ``$t\leq 0$'' on a functor category refer to the full subcategory of left $t$-exact functors.  We can ask: 

  \begin{enumerate}
      \item
  When is
  \[ \Fun^{L, t_{<0}}_{\QCsh(S)}(\QCsh(X), \QCsh(Y)) \subset \Fun^L_{\QCsh(S)}(\QCsh(X), \QCsh(Y)) \]
  a localization, so that this is in fact a $t$-structure?
\item What is the left completion / regularization of this $t$-structure?
  \end{enumerate}

  Suppose (i) held.  Then: the Theorem would imply that the $!$-integral transform induces an equivalence on regularization and completion, answering (ii).

  Let us consider (i): Note that $\theta$ is conservative and preserves all colimits.  If it commuted with filtered limits we would be done, and moreover the $t$-structure would be compatible with filtered colimits since the $t$-structure on the right hand side is and $\theta$ preserves colimits.  Under the identifications above, it suffices that the pushforward along the graph $\Gamma \colon X \times Y \to X \times S \times Y$
  \[ \Gamma_* \colon \QCsh(X \times Y) \to \QCsh(X \times S \times Y) \]
  (and all similar pushforwards along base-changes of graphs of $X \to S$, $Y \to S$, and diagonals of $S$) 
  be limit-preserving.  This is the case if and only if $f$ is of finite Tor dimension: in that case, $f^*$ preserves colimits and compact objects, and $f_*$ is its right adjoint.

  So, we see that it is necessary and sufficient to assume for (i) that $S$ be smooth.  In this case, $\QCsh S \isom \QC S$ is rigid and we have:
\end{remark}

\begin{corollary}\label{cor:dualiz-qcsh-with-t} Suppose that $S$ is smooth, and that $X$ and $Y$ are in addition of finite Tor dimension over $S$.  Then, there is a $t$-exact equivalence
  \[ \Fun^L_{\QCsh(S)}(\QCsh(X), \QCsh(Y)) \isom \ol{\QCsh}(X) \otimes_{\QCsh(S)} \QCsh(Y) \]
  where the left hand side is equipped with the $t$-structure of the previous Remark.

  The $!$-integral transform is left $t$-exact up to a finite shift, and realizes $\QCsh(X \times_S Y)$ as a left $t$-regularization of the functor category.
\end{corollary}

\section{Functors out of $\DCoh$}\label{sect: DCoh}
\begin{na} The standing assumptions for this section are the same as the previous section, except that in addition $S$ will be assumed \emph{smooth}.
\end{na}

We will now derive consequences, for small categories, of the results of the previous subsection.  Our main result will be:

\begin{theorem}\label{prop:fun-dcoh} Suppose that $S$ is a quasi-compact, geometric, and smooth $k$-stack over a perfect field $k$.  Suppose that that $\pi_X\colon X \to S$ is a quasi-compact finitely-presented separated $S$-algebraic space, and that $Y$ is a finitely-presented $S$-stack.  Then $!$-integral transforms give an equivalence
  \[ \Phish \colon 
  \DCoh_{prop/Y}(X \times_S Y/_! X) \stackrel\sim\longrightarrow \Fun^{ex}_{\Perf S}(\DCoh X, \DCoh Y)  \]
  where $\DCoh_{prop/Y}(X \times_S Y/_! X)  \subset \QCsh(X \times_S Y)_{<\infty}$ denotes
   the full subcategory consisting of those $\F \in \QCsh(X \times_S Y)$ satisfying:
  \begin{enumerate}
      \item $H_i \F = 0$ for $i \vinograd 0$;
      \item Each $H_i \F$ is coherent, with support proper over $Y$;
      \item $\F$ has finite-$\shotimes$ dimension over $X$: i.e., locally on $X \times_S Y$, the functor $\F \shotimes p_1^!(-)$ is right $t$-bounded up to finite shift;
  \end{enumerate}

  Furthermore: for $\K$ in this category, it is the case that $\DD \K$ is almost perfect and of finite Tor dimension over $X$, and there is a natural equivalence $\Phish_{K} = \DD \circ \Phi_{\DD \K} \circ \DD$.
\end{theorem}

Before giving the proof, we give a few remarks and elaborate the last sentence into a reformulation:
\begin{remark}
  The last sentence is asserting a \emph{different} relationship between $\Phish$ and $\Phi$ than that appearing in \autoref{cor:compare-phi-phish}.  In particular, the objects of $\QC(X \times_S Y)$ appearing here are not generally bounded above.  Let us describe them more explicitly:


Let 
$\DCoh_{prop/Y}(X\times_S Y/X)$ 
consist of the almost perfect complexes with support proper over $Y$ and which have finite ($*$--)Tor amplitude over $X$.  Then, there is a commutative diagram of equivalences
  \[ \xymatrix{
 \ar[d]^{\DD}|{\sim} \DCoh_{prop/Y}(X \times_S Y/_! X) \ar[r]^{\Phish}|{\sim} &  
 \Fun^{ex}_{\Perf S}(\DCoh X, \DCoh Y) \ar[d]^{\DD \circ (-) \circ \DD}|{\sim} \\
 \DCoh_{prop/Y}(X \times_S Y/X)^{op}  \ar[r]^{\Phi}|{\sim} & 
\Fun^{ex}_{\Perf S}(\DCoh X, \DCoh Y)^{op} 
  } \]  
  Note that the last arrow is an equivalence by unbounded Grothendieck duality \autoref{prop:groth-duality}: It shows that $\DCohp$ and $\DCohm$ are interchanged, and a separate boundedness argument shows that it preserves the property of having properly supported homology sheaves.\footnote{This is automatic for $\DCoh$, where compactly supported homology sheaves = compact support as a complex, since duality is compatible with open restriction and thus with support of a complex.  Now the property passes to left and right completions as before, noting that $H_i(\DD \F)$ depends on $\F$ only through a bounded truncation.}  It remains to check that ``finite $\shotimes$ dimension'' is interchanged with ``finite Tor dimension.''  Note that both sides are defined to be local on $X$ and on $X \times_S Y$.   This then follow from the following Lemma.
\end{remark}

\begin{lemma}\label{lem:tor-dimension} Suppose $f \colon \Spec A \to \Spec B$ is a map of connective Noetherian dg rings.  For $\F \in \DCohpm(A)$, the following are equivalent:
  \begin{enumerate}
      \item $\F \shotimes f^! -\colon \QCsh(B) \to \QCsh(A)$ is left and right $t$-bounded up to shift;
      \item $\F$ is homologically bounded above, and the functor $R (\F \shotimes f^! - ) \colon \QCsh(B) \to \QC(A)$ is right $t$-bounded up to shift;
      \item $\F$ is homologically bounded above, and the functor $R (\F \shotimes f^! -) \colon \DCoh(B) \to \DCohm(A)$ is right $t$-bounded up to shift;
      \item $\F$ is homologically bounded above, and the functor $R (\F \shotimes f^! (\DD -)) \colon (\DCoh(B))^{op} \to \DCohm(A)$ is right $t$-bounded up to shift;
      \item $\DD\F$ is homologically bounded below, and the functor $\DD\F \otimes f^*(-) \colon \DCoh(B) \to \DCohp(A)$ is left $t$-bounded up to shift;
      \item The functor $\DD\F \otimes f^* - \colon \QC(B) \to \QC(A)$ is left and right $t$-bounded up to shift.
  \end{enumerate}

  Note that (i) in our notation reads $\F \in \DCoh(A/_! B)$, while (vi) in our notation reads 
  $\DD\F \in \DCoh(A/B)$.
\end{lemma}
\begin{proof}
  Note that (i) implies that $\F$ is homologically bounded above, by taking $- = \omega_B$ and using left $t$-boundedness up to a shift.  Then, since $\shotimes$ is left $t$-bounded up to a shift, we see that (i) is equivalent to (ii) since $R \colon \QCsh(A)_{<n} \isom \QC(A)_{<n}$ is an equivalence for each $n$.

  Point (iii) is well-defined since $\shotimes$ and $f^!$ preserve the property of being in $\DCohm$ by \autoref{lem:groth-dual-functors} and the tensor product on $\DCohp$.  Then, (ii) is equivalent to (iii) since $\QCsh(B) = \Ind\DCoh(B)$ and the $t$-structure is compatible with filtered colimits.

  Continuing: (iii) is equivalent to (iv) by Grothendieck duality, including the boundedness assertion of \autoref{prop:dcoh-t-dual}. Next, (iv) is equivalent to (v) by another application of \autoref{lem:groth-dual-functors} and unbounded Grothendieck duality.  
  
  Finally, (vi) implies (v) since one sees that $\DD\F$ is homologically bounded below by taking $- = A$.  Conversely, (v) implies that the functor is $t$-bounded above on anything that is obtained as a filtered colimit of objects of $\DCoh(B)$ -- which includes all of $\QC(B)^{\heart}$ since $B$ is Noetherian, and all of $(\QC(B))_{<\infty}$ since $\QC B$ is right complete;  since $\DD\F$ is bounded below, the homology groups of 
  \[ \DD\F \otimes f^*(\tau_{\leq k}(-)) \] stabilize as $k \to \infty$ so that (v) is equivalent to (vi).
\end{proof}

The following is thus a reformulation of the Theorem, making no reference to shrieks.
\begin{corollary}\label{cor:functors-dcoh-reformulation}
Suppose that $\pi\colon X \to S$ are as in \autoref{prop:fun-dcoh}.  Then, star integral transforms give an equivalence
  \[ \DCoh_{prop/Y}(X \times_S Y/X) \stackrel\Phi\longrightarrow \Fun^{ex}_{\Perf S}(\DCoh X, \DCoh Y) \]
  where $\DCoh_{prop/Y}(X \times_S Y/X) \subset \QC(X \times_S Y)$ is as in the previous remark.
\end{corollary}

From this, we can deduce a few more Corollaries:
\begin{corollary} Restrict the above to the case of $Y = S$.  Then, we obtain that 
  \[ \Perf_{prop/S} X \stackrel{\sim}\longrightarrow \Fun^{ex}_{\Perf S}(\DCoh X, \DCoh S) \]
  via $\P \mapsto \Phi_{\P}$.
\end{corollary}
\begin{proof} By the reformulation of \autoref{cor:functors-dcoh-reformulation}, it is enough to recall that $\P$ perfect is equivalent to $\P$ almost perfect and of finite Tor amplitude.
\end{proof}

\begin{corollary} Restrict the above, and the results of \autoref{sec:perf} to the case of $Y = S$, and $X \to S$ proper.  Then, there are ``dualities''
  \[ \Fun^{ex}_{\Perf S}(\DCoh X, \DCoh S) \isom \Perf X \]
  \[ \Fun^{ex}_{\Perf S}(\Perf X, \DCoh S) \isom \DCoh X \]
\end{corollary}
\begin{proof} Combine the previous Corollary with the results of \autoref{thm: dcoh shriek}.
\end{proof}

The rest of this sction will be devoted to a proof of the Theorem.  The idea is to go in two steps:
\begin{enumerate} 
\item In \autoref{ssec:bounded}, we prove a generation result that will imply that any $\Perf S$-linear functor $\DCoh(X) \to \DCoh(Y)$ is automatically left (and right) $t$-bounded up to a finite shift.  This allows us to apply the results of the previous subsection to identity it with a subcategory of $\QCsh(X \times_S Y)$.
\item In \autoref{ssec:identify-C-sh}, we identify this subcategory.  The proof is similar to that presented in \autoref{sec:perf}.
\end{enumerate}

\subsection{A generation result, and a boundedness consequence}\label{ssec:bounded}
First, we need to establish a generation result:
\begin{lemma}\label{lem:very good2} Suppose that $\pi \colon X \to S = \Spec k$ is a quasi-compact, quasi-separated, and locally almost finitely presented algebraic space over a perfect field $k$.  Then, $\QCsh(X)$ is compactly-generated by $\DCoh(X)$, and there is a single $\G \in \Coh(X)^{\heart}$ which generates $\DCoh(X)$.
\end{lemma}
\begin{proof}\mbox{} This follows from \cite{DAG-XII}*{Theorem~1.5.10} applied to $\QCsh(-)$.  However to show that a single object suffices, a slightly different argument is required (since we do not a priori know a single generator for $\DCoh_Z(X)$ in the affine case).  We will prove this by Noetherian induction $X$, with the help of two observations:
 
  \medskip

  {\noindent}{\bf Claim 1: }{\it This property for $(X_{cl})_{red}$ implies this property for $X$.}\\

  {\noindent}{\bf Claim 2: }{\it Suppose $U \subset X$ is an open in a classical algebraic space $X$.  Let $Z$ denote the closed complement with its reduced induced structure.  Suppose that the property holds for each of $U$ and $Z$. Then it holds for $X$.}

  \medskip

  {\noindent}{\it Assuming the claims we complete the proof:}\\  
  We proceed by Noetherian induction, so suppose that the result is known for all proper closed sub-algebraic spaces $Z \subsetneq X$.  By Claim 1 we may suppose that $X$ is classical and reduced.  Since $k$ is perfect, we have that $X$ is generically smooth.  Thus there exists an open $U \subset X$ which is smooth and in particular regular.  Thus $\DCoh U = \Perf U$ so that \autoref{lem:very good} implies that the property holds for $U$.  By Noetherian induction, the result holds also for $Z = X \setminus U$ with its reduced induced structure so that Claim 2 completes the inductive step.

  \medskip

  {\noindent}{\it Proof of Claim 1:} \\ 
  Note that every $F \in \DCoh(\X)$ is an extension of the shifted sheaves $H_i \F$, which are pushed forward from $X_{cl}$.  And every $\F \in \DCoh(\X)^{\heart}$ is an interated extension of sheaves pushed forward from $(\X_{cl})_{\red}$ by filtering by powers of the nilradical ideal of $\O_{\X_{cl}}$.

  \medskip
  {\noindent}{\it Proof of Claim 2:} \\  
  Let $G_U \in \DCoh(U)$ and $G_Z \in \DCoh(Z)$ be objects whose shifts generate $\QCsh(U)$ and $\QCsh(Z)$.  Let $j \colon U \to X$ and $i \colon Z \to X$ be the inclusions.  The form of Thomason's argument in \cite{DAG-XI}*{Lemma~6.19} shows that there exists $G \in \DCoh(X)$ such that $j^* G \isom G_U \oplus G_U[+1]$.  Then, we claim that
  \[ G \oplus i_* G_Z  \in \DCoh(X) \]
  generates  $\QCsh(X)$ under shifts and colimits.  

  Suppose $\F \in \QCsh(X)$, and as before form the fiber sequence
  \[ \F_Z \longrightarrow \F \longrightarrow j_* j^* \F \]
 
  Consider the right adjoint $i^! \colon \QCsh_Z(X) \to \QCsh(Z)$ to $i_*$, so that $\F_Z = i_* i^! \F$.  Note that by passing to homology and filtering by the nilradical, as in the proof of Claim 1,  we see that $i_* \DCoh(Z)$ generates $\DCoh_Z(X)$ under cones, shifts, and retracts.  Thus, $i^!$ is conservative.  But,
  \[ 0 = \RHom_X(i_* G_Z, \F) = \RHom_Z(G_Z, i^! \F) = \RHom_Z(G_Z, i^! \F_Z) \]
  implies that $i^! \F_Z = 0$.  Thus, $\F_Z = 0$ and $\F = j_* j^* \F$.

  Next, exactly as in the proof of \autoref{lem:very good} observe that
  \[ 0 = \RHom_X(G, \F) = \RHom_X(G, j_* j^* \F) = \RHom_U(j^* G, j^* \F) \]
  which implies that $\RHom_U(G_U, j^* \F) = 0$ so that $j^* \F = 0$ and $\F = j_* j^* \F = 0$.
\end{proof}

Now we apply it to get a $t$-boundness result:

\begin{lemma}\label{lem:dcoh-bounded}
  Suppose that $S$ is a quasi-compact, finitely-presented, perfect $k$-stack; that $X$ is a quasi-compact, quasi-separated, and finitely-presented $S$-algebraic space; and that $Y$ is a quasi-compact and finitely-presented $S$-stack.

  Then every $\Perf(S)$-linear exact
  functor $F \colon \DCoh(X) \to \DCoh(Y)$ is left and right $t$-exact up to a shift.  In particular, Kan extension produces a fully faithful embedding
  \[ \Fun_{\Perf S}^{ex}(\DCoh X, \DCoh Y) \hookrightarrow \Fun_{\QC(S)}^{L, t-bdd}(\QCsh(X), \QCsh(Y)) \]
\end{lemma}
\begin{proof}
  The proof of \autoref{lem:fully-faithful}(i) shows that the inclusion $\DCoh Y \to \QCsh Y$ induces a fully-faithful functor
  \[ \Fun_{\Perf S}^{ex}(\DCoh X, \DCoh Y) \hookrightarrow \Fun_{\Perf S}^{ex}(\DCoh X, \QCsh Y) \] 
and using that $S$ is perfect and our unspoken assumption on $X$ we see that Kan extension identifies 
\[  \Fun_{\Perf S}^{ex}(\DCoh X, \QCsh Y) \stackrel{\sim}\longrightarrow \Fun_{\QC S}^{L}(\QCsh X, \QCsh Y) \]  
It remains to prove the uniform $t$-boundedness assertion.

  We proceed via a sequence of reductions:

  \medskip

  {\noindent}{\it Reduction to $Y$ affine: }
  First note that, since $Y$ is quasi-compact, the claim is flat local on $Y$.  So, we may suppose that $Y$ is affine.

  \medskip

  {\noindent}{\it Reduction to the absolute case: }
  Suppose that $f \colon U \to S$ is a smooth affine atlas.  Consider the diagram
  \[ \xymatrix{
  \DCoh X \ar[d] \ar[r]^F & \DCoh Y  \ar[d] \\
  \DCoh X \otimes_{\Perf S} \Perf U \ar[r]^-{F \otimes \id} \ar[d]^{\sim} & \DCoh Y \otimes_{\Perf S} \Perf U\ar[d]^{\sim} \\
  \DCoh (X \times_S U) \ar@{-->}[r]^{F_U}  & \DCoh (Y \times_S U)
  }
  \]
  We claim that the indicated vertical maps are indeed equivalences.  For this see e.g. \cite{indcoh}*{Prop.~4.5.3} and note that the proof applies in the present context.  Consequently, there is a unique functor $F_U$ making the diagram commute.  Suppose that $F_U$ is known to be left and right $t$-exact up to a shift.  It follows that the composite 
  \[ F_U \circ (f_X)^* = (f_Y)^* \circ F \colon \DCoh X \to \DCoh (Y \times_S U) \]
  is left and right $t$-exact up to a shift, since $f_X^*$ is a flat pullback map and hence $t$-exact.  By the previous reduction (that the claim can be checked flat locally on $Y$) this suffices, since $f_Y \colon Y \times_S U \to Y$ is a smooth cover.
  
  So, we may suppose $Y = \Spec A$ and $S$ are both affine and $X$ is a  finite-type separated algebraic space.  It suffices to show that any exact functor $F \colon \DCoh(X) \to \DCoh(Y)$ is left and right $t$-exact up to a shift: in particular, we no longer have any dependence on $S$.

  \medskip

  {\noindent}{\it Reduction to $X$ classical (in particular, of finite Tor dimension): } 
  Consider the closed immersion $i\colon X_{cl} \to X$ of the underlying classical algebraic space of $X$.  Note that $i_*\colon \DCoh(X_{cl}) \to \DCoh(X)$ is $t$-exact and induces an equivalence of the hearts; since the $t$-structure on $\DCoh X$ is bounded, it suffices to show that $ F \circ i \colon \DCoh(X_{cl}) \to \DCoh(Y)$ is left and right $t$-exact up to a shift.

  \medskip

  {\noindent}{\it Final proof:}
Observe that it follows from \autoref{prop:dcoh-t-dual} that  $F \colon \DCoh(X) \to \DCoh(Y)$ is right $t$-exact up to a shift if and only if $(\DD \circ F \circ \DD)^{op}$ is left $t$-exact up to a shift.  Consequently, it suffices to prove that any such $F$ is left $t$-exact up to a shift.

By \autoref{lem:very good2} there exists a $\G \in \DCoh(X)^{\heart}$ that generates.  Replacing $F$ by a shift we may suppose that $F(\G) \subset \DCoh(Y)_{\leq 0}$.  Setting $F' = \Ind F$, it suffices to show that any colimit preserving functor $F'\colon \QCsh X \to \QCsh Y$ satisfying 
\[ F'(\G) \in (\QCsh Y)_{\leq 0} \] 
is left $t$-bounded up to a shift.  Note that the full subcategory of left $t$-bounded functors is closed under cones, shifts, and retracts.  Note also that the colimit preserving functor 
\[ H' = \RHom_{\QCsh X}(\G, -) \otimes_k F'(\G) \]
is left $t$-exact.  Indeed, $\RHom_{\QCsh X}(\G, -)$ is left $t$-exact since $\G \in \DCoh(X)_{\geq 0}$; and $-\otimes_k F'(G)$ is left $t$-exact since $k$ is a field and $F'(G) \in \QCsh(Y)_{<0}$.  It thus suffices to show that $F'$ can be built from $H'$ in finitely many steps of taking cones, shifts, and retracts.  For this, it suffices to show that $\id_{\QCsh(X)}$ can be built from 
\[ \RHom_{\QCsh(X)}(\G, -) \otimes \G = \Phish_{\DD\GG \boxtimes \G} \] in finitely many steps.
  
By Grothendieck duality, $\DD \G \in \DCoh(X)$ also generates $\DCoh X$.  By \autoref{prop:absolute-fun} the exterior product $\boxtimes \colon \DCoh(X) \otimes \DCoh(X) \to \DCoh(X^2)$ is an equivalence, so that it follows that $\DD\G \boxtimes \G$ generates $\DCoh(X^2)$.   Since $\O_X$ is bounded, we have that $\omega_X \in \DCoh(X)$, so that $\Delta_* \omega_X \in \DCoh(X^2)$.  Consequently, $\Delta_* \omega_X$ may be built in finitely many steps, consisting of cones, shifts, and retracts, from $\DD\G\boxtimes \G$.  Applying the $!$integral transform $\Phish$ completes the proof since $\id_{\QCsh(X)} \isom \Phish_{\Delta_* \omega_X}$.
\end{proof}

\begin{corollary}\label{cor:fun-dcoh-0}  Suppose that $S, X, Y$ satisfy the hypotheses of \autoref{prop:fun-dcoh}. 
  Let \[ \sh{\C}_{X,Y/S} \subset \QCsh(X \times_S Y)_{<\infty} \] denote the full subcategory consisting of those $\K$ such that $\Phish_\K(\DCoh X) \subset \DCoh Y$.  Then, the restriction of the $!$-integral transform provides an equivalence 
        \[ \sh{\C}_{X,Y/S} \stackrel{\Phish}{\longrightarrow} \Fun_{\Perf S}^{ex}(\DCoh X, \DCoh Y) \]
\end{corollary}
\begin{proof} 
This is an immediate consequence of \autoref{lem:dcoh-bounded} and \autoref{cor:dcoh-t-reg}, identifying $\QC S \isom \QCsh S$ as symmetric monoidal categories since $S$ is smooth.
\end{proof}

\subsection{Identifying the right kernels}\label{ssec:identify-C-sh}
To complete the proof of the Theorem, we must identify $\sh{\C}_{X,Y/S}$ with $\DCoh_{prop/Y}(X \times_S Y/_! X)$.   We need the following preliminaries:
\begin{lemma}\label{lem:tensor-product:DDAPerf} In addition the assumptions of the Theorem, suppose that $Y = \Spec A$ is affine.  Let $\C \subset \DCohm(X \times_S Y)$ be the smallest full subcategory closed under finite limits, and containing $\DCohm(X \times_S Y)_{<0}$ and $p_1^!(\DCohm X)$.  Then, $\C = \DCohm(X \times_S Y)$.
\end{lemma}
\begin{proof} We first use unbounded Grothendieck duality \autoref{prop:groth-duality} to translate to a statement with usual pullbacks: Letting $\C' = \DD\C \subset \DCohp(X \times_S Y)$, it suffices to show that $\C' = \DCohp (X \times_S Y)$.  Note that $\C'$ is the smallest full subcategory closed under finite colimits, and containing $\DD(\DCohm(X \times_S Y)_{<0})$ and $\DD p_1^!(\DCohm X) = p_1^*(\DCohp X)$.  Furthermore, 
  \[ \DD(\DCohm(X \times_S Y)_{<0}) \supset \DCohp(X \times_S Y)_{>-N} \]
  for some $N$ by the boundedness assertion of \autoref{prop:groth-duality}.  
  
  It thus suffices to prove the following Grothendieck dual assertion: Let $\C' \subset \DCohp(X \times_S Y)$ be the smallest full subcategory closed under finite colimits, and containing $\DCohp(X \times_S Y)_{>0}$ and $p_1^*(\DCohp X)$.  Then, $\C' = \DCohp(X \times_S Y)$.

  \medskip
  
  {\bf Claim:} For each connective $\F \in \DCohp(X \times_S Y)_{\geq 0}$, there exist a connective $H \in \DCohp(X)_{\geq 0}$ and map  \[ \phi \colon p_1^* H \to \F \] which induces a surjection on $H_0$.

  \medskip

  Assuming the claim, we complete the proof.  We will show that $\DCohp(X \times_S Y)_{>-n} \subset \C'$ for all $n$, by induction on $n$. The case $n = 0$ is by hypothesis.  Let us prove the case of $n$, assuming known the case of $n-1$: 

  Suppose $\F \in \DCohp(X \times_S Y)_{>-n}$, so that $\F[n-1]$ is connective.  Apply the Claim to it, to obtain a map 
  \[ \phi \colon p_1^* H \to \F[n-1] \] inducing a surjection on $H_0$ as above.  Since $H$ is connective and $p_1^*$ is left $t$-exact, the $p_1^* H$ is connective.  Since $\phi$ induces a surjection on $H_0$, the fiber $\K = \fib(\phi)$ of $\phi$ is also connective.  Applying the Claim to $\K$ we obtain
  \[ \phi' \colon p_1^* H' \to \K \] inducing a surjection on $H_0$.  It follows that
  \[ \cone(p_1^* H' \to \K \to p_1^* H) \to \F[n-1] \]
  induces an isomorphism on $H_0$.  In light of the connectivity of all terms involved, it induces an isomorphism on $\tau_{\leq 0}$.  Consequently,
  \[ \cone(p_1^* H'[1-n] \to p_1^* H[1-n]) \to \F \]
  induces an isomorphism on $\tau_{\leq (1-n)}$.  In particular, the fiber is in $\DCohp(X \times_S Y)_{>1-n}$ and thus in $\C'$ by the inductive hypothesis.  Since $\C'$ is closed under cones, and contains $p_1^* H[1-n]$ and $p_1^* H'[1-n]$, we see that $\F \in \C'$.

  \medskip

  We now prove the claim: Since $p_1$ is affine, we have that $p_{1*} \F \in \QC(X)_{\geq 0}$.  Since $X$ is perfect, we can write
  \[ p_{1*} \F = \dlim_\alpha P_\alpha \]
  with $P_\alpha \in \Perf(X)$.  Since the $t$-structure is compatible with filtered colimits we have
  \[ p_{1*} \F = \dlim_\alpha \tau_{\geq 0} P_\alpha \]
  and since $X$ is Noetherian and each $P_\alpha$ is perfect, the truncations $\tau_{\geq 0} P_\alpha \in (\DCohp X)_{\geq 0}$ are almost perfect.

  Next, consider the composite
  \[ \dlim_\alpha p_1^* (\tau_{\geq 0} P_\alpha) \isom  p_1^* ( \dlim_\alpha \tau_{\geq 0} P_\alpha)  \stackrel\sim\longrightarrow p_1^* p_{1*} \F \longrightarrow \F \]
  Since $p_1$ is affine, and $\F$ is connective, the last map induces a surjection on $H_0$.  The previous displayed equation shows that $H_0(\F)$ is the increasing union of the images on $H_0$ of the terms for each $\alpha$.  Since $H_0(\F)$ is coherent, it is compact in $\DCoh(X \times_S Y)^{\heart}$.  Consequently, there is some $\alpha$ so that
  \[ p_1^* (\tau_{\geq 0} P_\alpha) \longrightarrow \F \]
  induces a surjection on $H_0$.  This completes the proof.
\end{proof}

\begin{lemma}\label{lem:2-t-structure}
Suppose that $\pi  \colon Z \to S = \Spec A$ is a separated finitely-presented $S$-algebraic space, over an affine (finitely-presented over $k$) base. Let $\ol{H}_i(\F)$, $\ol{\tau}_{\geq i}$, etc. be the homology and truncation functor for the dual $t$-structure on $\QCsh(Z)$ that was denoted $\ol{\QCsh}(Z)$ in \autoref{cor:dcoh-t-dual}.  Note that if $\F \in \DCoh(Z)$, then $\ol{H}_i(\F) = \DD \circ H_i \circ \DD(\F) \in \DCoh(Z)$ and that the functor $\ol{H}_i$ is determined by this equality and the property of preserving filtered colimits.
  
  Then, the following conditions on $\F \in \QCsh(Z)$ are equivalent:
  \begin{enumerate}
    \item $\F \in \QCsh(Z)_{<N}$ for some $N$, and $H_i(\F)$ is coherent over $Z$ with support proper over $S$ for each $i$, and they vanish for $i \vinograd 0$;
      \item $\F \in \DCohm(Z)$ and $H_i(\F)$ has support proper over $S$ for each $i$;
      \item $\F \in \DCohm(Z)$ and $\ol{H}_i \circ H_j(\F) = H_j\circ \ol{H}_i(\F)$ has support proper over $S$ for each $i,j$;
      \item $\F \in \DCohm(Z)$ and $\ol{H}_i(\F)$ has support proper over $S$ for each $i$;
      \item $\F \in \QCsh(Z)_{<N}$ for some $N$, and $\ol{H}_i(\F)$ is coherent over $Z$ with support proper over $S$ for each $i$, and they vanish for $i \vinograd 0$;
    \end{enumerate}
\end{lemma}
\begin{proof}
  Note that (i) is equivalent to (ii) by the definition of $\DCohm$.  
  
  Let us show that (ii)-(iv) are equivalent. First note note that for $\K \in \DCoh(Z)$ -- indeed for any bounded complex -- the support of $\K$, as a complex, is the union of the supports of its homology sheaves.  So if $\F \in \DCoh(Z)$ then $\F$ has proper support iff both $\F$ and $\DD \F$ have proper support iff each $H_i(\F)$ and $H_i(\DD \F)$ have proper support iff $\ol{H}_i(\DD\F) = \DD H_i(\F)$ and $\ol{H}_i(\F) = \DD H_i(\DD \F)$ all have proper support.  Since $\DD$ is $t$-bounded, one obtains that both (ii) and (iii) are both equivalent to requiring that $\DD H_i(\F)$ have support proper over $S$ for each $i$.

  Next, note that (iv) obviously implies (v).  For the converse, suppose that $\F \in \QCsh(Z)_{<N}$ is such that $\ol{H}_i(\F) \in \DCoh(Z)$ for all $i$.  Since $\F$ is bounded, and the identity functor between two $t$-structures is left $t$-exact up to a shift, it follows by induction $\ol{\tau}_{\geq i} \F \in \DCoh(Z)$ for all $i$.  Using that the identity functor is both left and right $t$-bounded up to a shift, it follows that $\tau_{\geq i} \F \in \DCoh(Z)$ for all $i$.  This completes the proof.
\end{proof}

We have the Grothendieck dual statement to \autoref{prop: push aperf}:
\begin{prop}\label{prop: push pscoh} 
Suppose that $\pi  \colon Z \to S = \Spec A$ is a separated finitely-presented $S$-algebraic space over an affine (finitely-presented over $k$) base.
Then, the following conditions on $\F \in \QCsh(Z)$ are equivalent:
  \begin{enumerate}
      \item $\F \in \DCohm(Z)$ and $H_i(\F)$ has support proper over $S$ for each $i$;
      \item $\F \in \QCsh(Z)_{<N}$ for some $N$, and $H_i(X, \H \shotimes \F) \in \DCoh(A)^{\heart}$ for all $\H \in \DCohm Z$;
      \item $\F \in \QCsh(Z)_{<N}$ for some $N$, and $\RGamma(\H \shotimes \F) \in \DCohm(A)$ for all $\H \in \DCohm Z$;
  \end{enumerate}
\end{prop}
\begin{proof}  Note that (ii) is equivalent to (iii) by the left $t$-exactness of the functors involved.  Furthermore, (i) implies (ii) by the proper pushforward Theorem and the left $t$-boundedness of the functor $H_i(X, \H \shotimes -)$ which allows us to replace $\F$ by something in $\DCoh(Z)$ whose support is proper over $S$.

  It suffices to show that (iii) implies (i).  We imitate the proof of \autoref{prop: push aperf}, making the same series of reductions:

  \medskip

  {\noindent}{\it Reduction to $Z$ and $A$ classical:}\\
  By \autoref{lem:2-t-structure} it suffices to show that $\ol{H}_i(\F)$ is coherent with proper support for all $i$.  Furthermore, the proof of op.cit. shows that these vanish for $i \vinograd 0$ since $\F$ is left bounded for the usual $t$-structure.  Without loss of generality we may suppose that $\ol{H}_i(\F) = 0$ for $i > 0$, and it will suffice -- since (i) implies (iii) -- to show that $\ol{H}_0(\F)$ is coherent with proper support.

Let $i\colon  Z_{cl} \to Z$ be the inclusion of the underlying classical algebraic space of $Z$.  Note that $i^!$ is left $t$-exact for the $\ol{\QCsh}$ $t$-structure, since $i_*$ is $t$-exact and $i^!$ is its right adjoint.  In contrast to the usual $t$-structure, the natural map
  \[ i_* i^! \ol{H}_0(\F)  \longrightarrow \ol{H}_0(\F)  \]
  induces an isomorphism on $\ol{H}_0$: Since everything commutes with filtered colimits, it suffices so check this assuming that $\ol{H}_0(\F) = \DD(\F')$ with $\F' \in \DCoh(Z)^{\heart}$.  Then, this is the map
  \[  \DD H_0 (i_* i^* \F') = \ol{H}_0(i_* i^! \F)    \longrightarrow    \ol{H}_0(\F) = \DD \F'   \] 
  Grothendieck dual to the usual map $\F \to i_* i^* \F$ which induces an equivalence on $H_0$.

  Thus, we may reduce to showing that $i^! \F$ satisfies the hypotheses of (i).  By the projection formula, it satisfies the hypotheses of (iii).  So, we are redued to showing that (iii) implies (i)  in case $Z$ (and $A$) is classical.

  {\noindent}{\it Reduction to $Z$ proper:}\\ Note that $\pi$ is assumed finite-type, so we no longer have to do that reduction.  The exact same Nagata compactification argument as before applies to reduce to the case of $Z$ proper.
  
  {\noindent}{\it Reduction to $Z$ projective projective:}\\ The same Chow's Lemma + Noetherian induction argument as before applies, now using the map $p_* p^! \F \to \F$, to reduce to the case of $Z$ projective.  The same argument as before reduces us to the case of $Z = \PP^n_S$ a projective space.
  
  {\noindent}{\it Case of $Z = \PP^n_S$:}\\ In this case, we can build the identity functor on $\QC(Z \times_S Z)$ out of functors of the form 
  \[ \P' \otimes_A  \RGamma(Z, \P \otimes -) = \P' \otimes_A \RGamma(Z, (\omega_Z \otimes \P) \shotimes -) \]
  so that the hypotheses on $\F$ imply that it lies in $\DCohm Z$.  This completes the proof.
\end{proof}

Finally, we're ready to complete the proof of the Theorem:
\begin{proof}[Proof of \autoref{prop:fun-dcoh}]
  We must characterize those $\K \in \QCsh(X \times_S Y)_{<\infty}$ such that $\Phish_{\K}(\DCoh X) \subset \DCoh Y$. Note that the characterization, and the hypothesis, are both local on $Y$, so that we may suppose $Y$ affine.  Since $\Phish_{\K}$ is left $t$-bounded, it follows by an approximation argument that $\Phish_{\K}(\DCohm X) \subset \DCohm Y$.

  Let $\C \subset \DCohm(X \times_S Y)$ denote the full subcategory consisting of those $\F$ such that
  \[ H_i\left( p_{2_*}(\F \shotimes \K)\right) \in \DCoh Y^{\heart} \quad \text{for all $i \leq 0$} \]
  Since $\K$ is bounded above, there is some $N$ such that $\DCohm(X \times_S Y)_{<-N} \subset \C$.  It is closed under finite limits by inspection, and contains $p_1^!(\DCohm X)$ since $\Phish_{\K}(\DCohm X) \subset \DCohm Y$ by the above.  Thus by \autoref{lem:tensor-product:DDAPerf} we have $\C = \DCohm(X \times_S Y)$.

  Consequently, noting that the functor is left $t$-bounded up to a shift, we have that
  \[ p_{2*}(\F \shotimes \K) \in \DCohm Y \quad \text{for all $\F \in \DCohm(X \times_S Y)$}. \]

  By \autoref{prop: push pscoh} any such $\K$ lies in $\DCohm(X \times_S Y)$ and has homology sheaves that are compactly supported over $Y$.  It remains to check that assertion about finite $\shotimes$-amplitude over $X$.  Note that the functor
  \[ \RGamma(X \times_S Y, \K \shotimes (p_1)^!(-))  = \RGamma(Y, \Phish_{\K}(-)) \]
  is left and right $t$-bounded by \autoref{lem:dcoh-bounded} and the finite cohomological dimension of $Y$.  If $X$ were affine, we would thus be done by \autoref{lem:tor-dimension}.  For the general case, let $q\colon U \to X$ be a smooth cover by an affine scheme and let $q'\colon U \times_S Y \to X \times_S Y$ be its base-change.  
  
  Note that $(q')^! \K$ has finite $\shotimes$-dimension over $U$: Indeed, for $\F \in \QCsh U$ base-change and projection provide equivalences
  \[ \RGamma(X \times_S Y, \K \shotimes p_1^! q_* \F) = \RGamma\left(X \times_S Y, \K \shotimes q'_*(p_1^!(\F))\right) = \RGamma\left(U \times_S Y, (q')^! \K \shotimes p_1^! \F \right)  \]
  so that this follows by the above, together with the observation that $q_*$ is left and right $t$-bounded up to a shift.

  Since $q'$ is smooth, this also implies that $(q')^* \K$ has finite $\shotimes$-dimension over $U$, which was our definition.  This completes the proof.
\end{proof}

\subsection{Complements: Functors out of $\DCoh \otimes_{\Perf} \DCoh$}
We have now described the case of functors out of $\Perf X$ and $\DCoh X$.  In case $X$ is of finite Tor dimension, there is a natural restriction functor between them and we have seen that the descriptions of functor categories are compatible with this.  In case $X$ is \emph{quasi-smooth}, there are also a variety of categories between the two.  There should likely be a similar result for functor categories out of these $\DCoh_{\Lambda}$ in general.  We will, however, content ourselves with certain special subcategories: Those gotten as the essential image of exterior products $\DCoh \otimes_{\Perf} \DCoh \otimes_{\Perf} \cdots$.

\begin{lemma} 
Suppose $S$ is a quasi-compact, geometric, smooth $k$-stack; that $X_i \to S$ are quasi-compact, separated, finitely-presented, finite Tor-dimension $S$-algebraic spaces for $i = 1, \ldots, n$; and that $Y$ is a quasi-compact and finitely-presented $S$-stack.  Then, the restriction
  \[ \Fun^{ex}_{\Perf S}(\DCoh X_1 \otimes_{\Perf S} \cdots \otimes_{\Perf S} \DCoh X_n, \DCoh Y) \longrightarrow \Fun^{ex}_{\Perf S}(\Perf X_1 \otimes_{\Perf S} \cdots \otimes_{\Perf S} \Perf X_n, \DCoh Y) \] 
  is fully faithful.  Furthermore, one recovers $\F$ from the restriction $i^* \F$ by Kan extending and restricting along $\DCoh \hookrightarrow \QC$.
\end{lemma}
\begin{proof} It will suffice to prove the following two claim:
 
  \bigskip 

  {\noindent}{\bf Claim 1:} {\em Suppose $F \in \Fun^{ex}_{\Perf S}(\DCoh X_1 \otimes_{\Perf S} \cdots \otimes_{\Perf S} \DCoh X_n, \DCoh Y)$.  Then, $F$ is ``right $t$-exact up to a shift'' in the sense that there exists an integer $N$ so that 
  \[ F\left( (\DCoh X_1)_{>0} \times \cdots \times (\DCoh X_n)_{>0} \right) \subset (\DCoh Y)_{>-N} \]}

  Note that if $n=1$, this was \autoref{lem:dcoh-bounded}.  As there, conjugating with Grothendieck duality implies that any such $F$ is also left $t$-exact up to shift in the appropriate sense.

\bigskip

Consider the exterior product
\[ \boxtimes_S \colon \DCoh(X_1) \otimes_{\Perf(S)} \cdots \otimes_{\Perf(S)} \DCoh(X_n) \longrightarrow \DCoh(X_1 \times_S \cdots \times_S X_n) \]

  {\noindent}{\bf Claim 2:} {\em Suppose $F \in \Fun^{ex}(\DCoh(X_1 \times_S \cdots \times_S X_n)), \QC Y)$ is such that $F \circ \boxtimes_S$ is right $t$-exact up to a shift in the sense of Claim 1.  Then, $F$ is right $t$-exact up to a shift in the usual sense.}

\bigskip

  Let us explain how the claims imply the desired result.
 Let us introduce the temporary notations $$\DCoh(X_I/S)= \DCoh X_1 \otimes_{\Perf S} \cdots \otimes_{\Perf S} \DCoh X_n$$ and
 $$\Perf(X_I/S)=\Perf X_1 \otimes_{\Perf S} \cdots \otimes_{\Perf S} \Perf X_n.$$
 Now consider the diagram
\[  \xymatrix{
\Fun^{ex}_{\Perf S}(\DCoh(X_I/S), \DCoh Y) \ar[r]  \ar@{^{(}->}[d] & \Fun^{ex}_{\Perf S}(\Perf(X_I/S), \DCoh Y)  \ar@{^{(}->}[d] \\
\Fun^{ex}_{\Perf S}(\DCoh(X_I/S), \QC Y) \ar[r] \ar@{^{(}->}[d]_{\mathrm{LKan}_{\boxtimes_S}} & \Fun^{ex}_{\Perf S}(\Perf(X_I/S), \QC Y) \ar@{^{(}->}[d]^{\sim}_{\mathrm{LKan}_{\boxtimes_S}}  \\
\Fun^{ex}_{\Perf S}(\DCoh(X_1 \times_S \cdots \times_S X_n), \QC Y) \ar[r] & \Fun^{ex}_{\Perf S}(\Perf(X_1 \times_S \cdots \times_S X_n), \QC Y) \\
}\]   
where the top set of vertical arrows are induced by the inclusions, and the bottom set of vertical arrows are given by left Kan extension along the respective fully faithful exterior product functor.  From the diagram, we see that it is thus enough to show that bottom arrow -- induced by restriction along the inclusion of $\Perf$ into $\DCoh$ -- is fully faithful on the essential image of the left vertical composite.  The Claims, taken together, imply that this essential image consists entirely of functors which are right $t$-exact up to a shift.  

Let
\[ L \colon \QC(X_1 \times_S \cdots \times_S X_n) \longrightarrow \QCsh(X_1 \times_S \cdots \times_S X_n) \]  be the Kan extension of $\Perf(\ldots) \to \DCoh(\ldots)$.
It is enough to show that the restriction functor
\[ L^* \colon \Fun^{L}_{\QC S}(\QCsh(X_1 \times_S \cdots \times_S X_n), \QC(Y)) \longrightarrow \Fun^{L}_{\QC S}(\QC(X_1 \times_S \cdots \times_S X_n), \QC(Y)) \]
is fully faithful on the subcategory of right $t$-exact functors.

To do this, we introduce some more notation: Note that $L$ preserves compact objects and colimits, so that it admits a limit-and-colimit preserving right adjoint $M$.  Furthermore, $M$ is $t$-exact.  Consequently, $L^*$ is right adjoint to $M^*$ and it suffices to show that $M^* L^*(F) \isom F$ for a right $t$-exact functor $F$.  

But recall that $L$ and $M$ induces an equivalence on co-connective objects.  In particular, for every $V \in \QCsh(...)$ the co-unit map
$L(M(V)) \to V$ has infinitely connective cone.  Consequently, any right $t$-exact functor carries it to an infinitely connective object of $\QC(Y)$ -- since the latter is left $t$-complete, any infinitely connective object is a zero object.  This proves that the counit $M^* L^*(F) \to F$ is an equivalence, and completes the proof modulo the Claims.

  \bigskip

  {\noindent}{\it Proof of Claim 1: }\\ 
  Note that $F$ gives rise to a $\Perf S^{\otimes n}$-linear functor
  $ F' \colon \DCoh X_1 \otimes \cdots \otimes \DCoh X_n \longrightarrow \DCoh Y$ 
  and that $F$ satisfies the conclusion of the claim if and only if $F'$ satisfies it.  Consider the commutative diagram
  \[ \xymatrix{
  \DCoh X_1 \otimes \cdots \otimes \DCoh X_n \ar[r]^-{F'} \ar[d]^{\sim}_{\boxtimes} & \DCoh Y \\
  \DCoh (X_1 \times \cdots \times X_n) \ar@{-->}[ur]_-{F''}
  } \]
  Since the vertical map is a $\Perf S^{\otimes n}$-linear equivalence, there exists a $\Perf S^{\otimes n}$-linear functor $F''$ making the diagram commute.  By \autoref{lem:dcoh-bounded}, $F''$ is right $t$-exact up to a shift in the usual sense.  Since the exterior product is right $t$-exact in the sense of the Claim, this completes the proof.

  \bigskip

  {\noindent}{\it Proof of Claim 2: }\\ We may pass to large categories, so
suppose that 
\[ F \colon \QCsh(X_1 \times_S \cdots \times_S X_n) \longrightarrow \QC Y \]
is a colimit-preserving functor such that
\[ F(V_1 \boxtimes_S \cdots \boxtimes_S V_n) \in \QC(Y)_{>0} \qquad \text{for all tuples with $V_i \in \QCsh(X_i)_{>0}$} \]
We wish to show that $F$ is right $t$-exact up to a finite shift.

\medskip

{\it Reduction to $S$ affine:} Suppose $p \colon U \to S$ is a smooth cover by an affine scheme.  Let $p_i \colon X'_i = X_i \times_S U \to X_i$ be the projections.   Imitating the proof of \autoref{lem:dcoh-bounded}, we see that we can reduce to the affine case provided we check that $F_U$ is right $t$-exact provided that $F$ is.  That is, we are given that
\[ F_U(p_i^* V_1 \otimes \cdots \otimes p_n^* V_n) \] is connective for $V_i$ connective, and we must conclude that it is also the case that
$ F_U(V'_1 \otimes \cdots \otimes V'_n)$ is connective for arbitrary $V'_i$ connective.  For this, we note that connective objects are closed under colimits, that tensor and $F_U$ all preserve colimits, and the geometric realization diagram
\[ V'_i \stackrel\sim\longleftarrow \left\| (p_i^* p_{i*})^{\bullet+1} V'_i \right\| \]
Note that each term of the geometric realization is the pullback of a connective object, since $p_i$ is smooth and affine so that $p_i^*$ and $p_{i*}$ are both $t$-exact.  This completes the reduction to $S$ affine.

\medskip

{\it Reduction to $S = Y = \pt$:} The map 
\[ i \colon X_1 \times_S \cdots \times_S X_n  \longrightarrow X_1 \times \cdots \times X_n \] has finite Tor dimension since $S$ is smooth.  Thus, there is a well-defined -- and right $t$-exact -- pullback functor 
\[ i^* \colon \QCsh(X_1 \times \cdots \times X_n) \longrightarrow \QCsh(X_1 \times_S \cdots \times_S X_n) \]
and furthermore
\[ i^*(V_1 \boxtimes \cdots \boxtimes V_n) = V_1 \boxtimes_S \cdots \boxtimes_S V_n \] for tuples as above.  Thus, it is enough to replace $S$ by $\pt$ and $F$ by $F \circ i^*$.  Thus we are in the absolute setting $S = \pt$.  Note also that the claim is local on $Y$, so we may suppose that $Y$ is affine.  Composing with the exact, conservative, global sections functor we may suppose that $Y = \pt$.

\medskip

{\it Reduction to $S = \pt$ affine:} We now handle a series of increasingly more complicated cases.

\smallskip

Note first that if all the $X_i$ are \emph{affine}, then the claim is easy: In this case, $\QCsh(X_1 \times_S \cdots \times_S X_n)_{\geq 0}$ is just generated by the structure sheaf $\O$, which is an exterior product of structure sheaves.  Thus, we can conclude that $F$ is in fact right $t$-exact.

\smallskip

Next, for pedagogical purposes, suppose that $X_2, \ldots, X_n$ are affine, while $X_1 = U \cup V$ is the intersection of two affine schemes.  Set $\X = X_1 \times \cdots \times X_n$, and $\X_U = \X \times_{X_1} U$ and similarly for $\X_V$ and $\X_{U \cap V}$.  Let $j_U \colon \X_U \to \X$ be the inclusion, and similarly for $V$.  Since $(j_U)_*$ is $t$-exact, we see that the above case applies to $F \circ j_U$  -- thus $F \circ j_U$ is right $t$-exact, and similarly for $F \circ j_V$ and $F \circ j_{U \cap V}$.  Now, for any $\G \in \QCsh(\X)$ we have a (rotated) Mayer-Vietoris cofiber sequence
\[ (j_{U*} j_U^* \G \oplus j_{V*} j_V^* \G)[-1] \longrightarrow (j_{U\cap V*} j_{U \cap V}^* \G)[-1]  \longrightarrow \G\]
Applying $F$ we obtain
\[ (F \circ j_{U*}(j_U^* \G) \oplus F \circ j_{V*}(j_V^* \G))[-1] \longrightarrow (F \circ j_{U\cap V*}(j_{U \cap V})^* \G)[-1]  \longrightarrow F(\G) \]
and since the first two terms are in $\QC(Y)_{\geq -1}$ by the above, we see that $F(\G) \in \QC(Y)_{\geq -1}$.   Thus, $F$ is right $t$-exact up to a shift by $1$.  A similar argument, by inducting on the number of affine opens in a cover of each $X_i$, completes the proof in case where each $X_i$ is a quasi-compact and separated scheme.

\smallskip

Suppose now that $X_2, \ldots, X_n$ are schemes, while $X_1$ is an algebraic space.  Notice that if $f \colon X'_1 \to X_1$ is any affine morphism from a scheme or algebraic space for which we know the result, then $F \circ f_*$ is right $t$-exact up to a shift by the above cases.   Consider now an excision square as before: So $U \subset X_1$ is an open subspace for which we know the result, $Z$ is its closed complement, and there is an \'etale map $\eta \colon \Spec R \to \X_1$ such that $\eta^{-1}(U)$ is affine.  Then, $\eta^{-1}(Z)$ is cut out by some $f_1, \ldots, f_d \in H_0(R)$.  In this case, for any $\G \in \QCsh(\X)_{\geq 0}$ we can consider the rotated cofiber sequence
\[ j_{U*} j_U^* \G[-1] \longrightarrow \F_Z \longrightarrow \G \]
and applying $F$ we obtain
\[ (F \circ j_{U*}) j_U^* \G[-1] \longrightarrow F(\F_Z) \longrightarrow F(\G) \]
We can bound the connectivity of the first term by using that we know the result for $U$; we bound it for the second term by computing the local cohomology on $\Spec R$ and using the explicit Koszul sequence (to get a bound of $(-d)$-connective).  This boundes the connectivity of $F(\G)$.  This completes the proof in this simplified case, and the general case is analogous.
   \end{proof}

 \begin{prop}\label{prop:fun-dcoh-many}
  Suppose that $S$ is a regular Noetherian perfect stack; that $X_i \to S$ are finite-type, finite Tor-dimension, relative algebraic spaces for $i=1,\ldots,n$; and that $Y$ a perfect $S$-stack.
  
  Let $\C_{X_i,Y/S} \subset \QC(X_1 \times_S \cdots \times_S X_n \times_S Y)$ denote the full subcategory of those $\K$ such that 
  \[ \Phi_{\K}(\F_1 \boxtimes_S \F_2 \cdots \boxtimes_S \F_n) \in \DCoh Y \qquad \text{for all $\F_i \in \DCoh(X_i)$} \] 
  
  Then: $\C_{X_i,Y/S} \subset \DCoh_{prop/Y}(X_1 \times_S \cdots \times_S X_n \times_S Y)$ and the restriction of the star integral transform induces an equivalence
  \[ \C_{X, Y/S} \stackrel{\Phi}\longrightarrow \Fun^{ex}_{\Perf S}(\DCoh X_1 \otimes_{\Perf S} \DCoh X_2 \otimes_{\Perf S} \cdots \otimes_{\Perf S} \DCoh X_n, \DCoh Y) \]
\end{prop}
\begin{proof} This follows immediately from the previous Lemma and \autoref{thm: dcoh shriek}.  Indeed, it identifies the functor category with a full subcategory of 
  \begin{align*} \Fun^{ex}_{\Perf S}(\Perf X_1 \otimes_{\Perf S} \cdots \otimes_{\Perf S} \Perf X_n, \DCoh Y)  &\isom \Fun^{ex}_{\Perf S}(\Perf (X_1 \times_S \cdots \times_S X_n), \DCoh Y)  \\
    &\isom \DCoh_{prop/Y}(X_1 \times_S \cdots \times_S X_n \times_S Y) 
  \end{align*}
  and it follows from the last phrase in the Lemma that this is the indicated subcategory.  Finally, recall that the proof of \autoref{thm: dcoh shriek} in fact identified $\DCoh_{prop/Y}(\cdots)$ as the full-subcategory of $\QC(\cdots)$ consisting of those $\K$ such that $\Phi_\K(\Perf) \subset \DCoh$.
\end{proof}


\section{Appendix: Recollections on $t$-structures}\label{app:t-bdd}
The goal of this Appendix is to recall some constructions having to do with $\infty$-categories with $t$-structure which are implicit in many places, and possibly explicit in some, in this paper.  Some constructions similar to $\R$ appear in \cite{FrenkelGaitsgory-Dmod}*{Section~22} and ideas similar to those exposed here have also been worked out by J. Lurie in unpublished work.  The present exposition is a shortened version of the Appendices to a not-yet-available preprint of the third author \cite{toly-loops}, so we will omit some proofs.

\subsection{Completions of $t$-structures}
For the reader's convenience, we recall several convenient conditions and constructions with $t$-structures from \cite{LurieHA}:
\begin{lemma} Suppose $\C$ is a stable presentable $\infty$-category with accessible $t$-structure (recall \cite[Def.~1.4.5.12]{LurieHA} that this is equivalent to requiring $\C_{<0}$ to be presentable).  Then, TFAE:
	\begin{enumerate}
		\item $\C_{<0}$  is closed under filtered colimits in $\C$.
		\item $i_{<0}\colon \C_{<0} \to \C$ preserves filtered colimits.
		\item $L_{<0} = i_{<0} \tau_{<0}\colon \C \to \C$ preserves filtered colimits.
		\item $L_{\geq 0} = i_{\geq 0} \tau_{\geq 0} \colon \C \to \C$ preserves filtered colimits.
		\item $\tau_{\geq 0}\colon \C \to \C_{\geq 0}$ preserves filtered colimits.
	\end{enumerate}
	These equivalent conditions imply that:
	\begin{enumerate}
		\item $i_{\geq 0}\colon \C_{\geq 0} \to \C$ preserves compact objects.
		\item $\tau_{<0}\colon \C \to \C_{<0}$ preserves compact objects.
	\end{enumerate}
	Furthermore,
	\begin{itemize}
		\item If $\C$ is compactly-generated, then so is $\C_{<0}$ (with compact objects retracts of objects of the form $\tau_{<0} \K$, $\K \in \C^c$).  In this case, the above conditions are \emph{equivalent to} $\tau_{<0}$ preserving compact objects.
		\item If $\C$ and $\C_{\geq 0}$ are compactly-generated, then the above conditions are \emph{equivalent} to $i_{\geq 0}$ preserving compact objects.
	\end{itemize}
\end{lemma}
\begin{proof} Omitted.
\end{proof}

\begin{defn} Suppose $\C$ is a stable $\infty$-category with $t$-structure.  We say that the $t$-structure is \demph{compatible with filtered colimits} if $\C$ has all filtered colimits and $\C_{<0}$ is closed under filtered colimits in $\C$.
\end{defn}

\begin{defn} Suppose $\C$ is a stable $\infty$-category with $t$-structure.  
	\begin{itemize}
		\item We say that the $t$-structure is (weakly) \demph{left complete} if the natural map
			\[ \F \longrightarrow \ilim_n \tau_{<n} \F \] is an equivalence for all $\F \in \C$ (in particular, the inverse limit is required to exist).
                      
                        We say that it is left complete if furthermore every tower in $\ilim \C_{<n}$ comes from an object of $\C$.
		\item We say that the $t$-structure is (weakly) \demph{right complete} if the natural map
			\[\dlim_n \tau_{\geq n} \F \longrightarrow \F \] is an equivalence for all $\F \in \C$ (in particular, the direct limit is required to exist).
                        
                        We say that it is right complete if furthermore every diagram of objects in $\ilim \C_{>-n}$ comes from an object of $\C$.
	\end{itemize}
\end{defn}

\begin{remark} The previous definition is of course formally symmetric: a $t$-structure on $\C$ is left complete iff the opposite $t$-structure on $\C^\op$ is right complete.  In practice there is however a substantial asymmetry: We are generally interested in presentable categories, and the opposite of a presentable category is almost never presentable.  More practically, the categories that arise in algebraic geometry -- at least for our purposes -- tend to be right-complete, but some interesting categories fail to be left-complete.  
\end{remark}

\begin{remark} 
By \cite{LurieHA}*{1.2.1.19}, this distinction between the ``(weakly)'' and not variants disappears for the notion of left complete (resp., right complete) provided that $\C$ has countable products (resp., coproducts), and that countable products are right $t$-exact (resp., coproducts are left $t$-exact) up to a finite shift.
  
In particular, if $\C$ is presentable and the $t$-structure compatible with filtered colimits then ``weakly left complete'' coincides with ``left complete.''
\end{remark}

\begin{example} Suppose $A \in \Alg(k\mod)$.  Then, $A\mod$ is equipped with a right complete accessible $t$-structure compatible with filtered colimits.  It is uniquely characterized by the following: $(A\mod)_{>0}$ is generated by those $M \in A\mod$ whose underlying complex is connective, that is lies in $(k\mod)_{>0}$.  Then, $(A\mod)_{\leq 0}$ is recovered as the right-orthogonal to this.  It follows that the forgetful functor $A\mod \to k\mod$ is right $t$-exact; it is left $t$-exact iff $A$ is connective:
	\begin{itemize}
		\item Suppose $A$ is itself connective.  Then, $(A\mod)_{<0}$ consists of those $M \in A\mod$ whose underlying complex is co-connective, that is lies in $(k\mod)_{<0}$.  In this case, the $t$-structure is left complete.
		\item Suppose $A=C^*(BS^1, k)\isom k\ps{\bt}$, where $\bt$ is in homological degree $-2$.  One can explicitly describe the $t$-structure in this case as follows: $(A\mod)_{>0}$ is generated by $k[+n]$, $n > 0$, so that $(A\mod)_{<0}$ consists of those $M$ for which $\RHom_A(k, M) \in (k\mod)_{<0}$.  Consequently, the $t$-structure is not left complete: $k\pl{\bt}$ is a non-zero object which is in $(A\mod)_{<n}$ for all $n$ since $\RHom_A(k, k\pl{\bt}) = 0$.
	\end{itemize}
\end{example}

\subsection{Coherent and Noetherian $t$-structures}
Assuming some extra ``finiteness'' conditions on the $t$-structure, one has extra operations of \emph{regularization} available in addition to \emph{completion}.

\begin{lemma} Suppose $\C$ is a stable $\infty$-category with $t$-structure that is compatible with filtered colimits.  For $\F \in \C$, the following conditions are equivalent
	\begin{enumerate}
		\item $\tau_{<n} \F \in \C_{<n}$ is compact for all $n \in \ZZ$;
		\item $\Map_{\C}(\F, -)$ commutes with filtered colimits in $\C_{<n}$ for all $n \in \ZZ$ (``commutes with uniformly bounded above colimits'');
		\item $\RHom_{\C}(\F, -)$ commutes with filtered colimits in $\C_{<n}$ for all $n \in \ZZ$.
	\end{enumerate}

	Furthermore,
	\begin{itemize}
		\item Suppose in addition that $\F$ is assumed bounded above: $\F \in \C_{<n}$.  Then, the above are equivalent to: $\F \in (\C_{<n})^c$ and its image under the inclusion $i_{<n}\colon \C_{<n} \to \C_{<m}$ is compact for all $m \geq n$;
		\item Suppose that $\C$ is right complete.  If $\F$ is bounded above and satisfies the above equivalent conditions, then it is also bounded below.
	\end{itemize}
\end{lemma}
\begin{proof}
Omitted.
\end{proof}

\begin{defn} Say that $\F \in \C$ is \demph{almost compact} if $\F$ satisfies the equivalent conditions of the previous Lemma.  Say that $\F \in \C$ is \demph{coherent} if 
	\begin{enumerate}
		\item $\F$ is bounded above, i.e., $\F \in \C_{<n}$ for some $n$.
		\item $\F$ satisfies the equivalent conditions of the previous Lemma.
	\end{enumerate}
	(If $\C$ is right complete, then any such $\F$ is also bounded below by the previous Lemma.)

        Define the full subcategory \demph{$\Cohp(\C) \subset \C$ (resp., $\Coh(\C) \subset \C$)} to consist of all $\F \in \C$ that are almost compact (resp., coherent).\footnote{This notation is potentially confusing, but fortunately will not be used much in general: $\Cohp(\C)$ need not be the left $t$-completion of $\Coh(\C)$ in general.}
      \end{defn}
  
\begin{remark}
Characterization (iii) of the previous Lemma makes clear that $\Cohp(\C)$ and $\Coh(\C)$ are stable subcategories.  Notice that, in general, the the $t$-structure need not restrict to these subcategories.
\end{remark}

We can impose the following more stringent conditions to eliminate this issue:
\begin{lemma}\label{lem:t-works} Suppose $\C$ is a stable $\infty$-category with $t$-structure that is compatible with filtered colimits.  Then, the following conditions are equivalent:
  \begin{enumerate}
      \item The $t$-structure on $\C$ restricts to one on $\Cohp(\C)$;
      \item The truncation functors on $\C$ preserves $\Cohp(\C)$.
      \item The inclusion $i_{<0}\colon \C_{<0} \to \C_{<1}$ preserves compact objects;
      \item The loops functor $\Omega\colon \C_{<0} \to \C_{<0}$ preserves compact objects;
    \end{enumerate}
  In this case, $\Cohp(\C)^{\heart} = \Coh(\C)^{\heart} = (\C^{\heart})^c$.

    These imply -- and in case $\C$ is right complete, are equivalent to --
\begin{enumerate}[resume]
      \item The subcategory of compact-objects in the heart $(\C^{\heart})^c \subset \C^{\heart}$ is abelian;
  \end{enumerate}

\end{lemma}
\begin{proof}
  Note that (i) $\Leftrightarrow$ (ii) is clear.  It is easy to check that (ii) $\Leftrightarrow$ (iii): It is enough to note that $\tau_{\leq k} \F \in \C_{\leq k}$ implies that $\tau_{\leq k'} \F \in \C_{\leq k'}$ is compact for all $k' \leq k$, and shifts of (iii) give the rest.  Finally (iii) $\Leftrightarrow$ (iv) since we may identify the two functors under the idenfication $\C_{<1} \isom \C_{<0}[1]$. 

  Assuming (i)-(iv), it is clear that $\Coh(\C)^{\heart}$ consist precisely of the compact objects of $\C^{\heart}$.

  Finally, note that (iii) clearly implies (v). If $\C$ is right complete, the compact objects of $\C_{<0}$ are bounded, giving the converse.
\end{proof}

Under the above hypotheses, we have:
\begin{lemma}  Suppose $\C$ is a stable $\infty$-category with $t$-structure that is compatible with filtered colimits, right complete, and satisfies the equivalent conditions of \autoref{lem:t-works}.  Then:
  \begin{itemize}
    \item $\Coh(\C)^{\heart} = \Coh(\C) \cap \C^{\heart} = (\C^{\heart})^c$ consists precisely of the compact (or ``finitely presented'') objects of $\C^\heart$.
        \item $\F \in \Cohp(\C)$ if and only if $H_n \F \in \Coh(\C)^{\heart} \subset \C^\heart$ and $H_n \F = 0$ for $n \ll 0$;
        \item $\F \in \Coh(\C)$ if and only if $H_n \F \in \Coh(\C)^{\heart} \subset \C^\heart$ and $H_n \F = 0$ for all but finitely many $n$.
  \end{itemize}  
\end{lemma}
\begin{proof} Omitted.
\end{proof}

Finally, we come to a strengthening of the above:
\begin{lemma} Suppose that $\C$ is a stable $\infty$-category with $t$-structure that is compatible with filtered colimits, that is right complete and that satisfies any of the equivalent conditions of  \autoref{lem:t-works}.
  
  Then, the following conditions are equivalent:
\begin{enumerate}
  \item $\C_{<0}$ is compactly-generated as $\infty$-category (in particular, presentable);
  \item $\C^{\heart}$ is compactly-generated as ordinary category;
  \item $\C^{\heart}$ is a locally coherent abelian category.  (Recall this means that the compact objects form an abelian category, and that $\C^{\heart}$ is compactly generated.  In particular, it is Grothendieck.)
\end{enumerate}
\end{lemma}
\begin{proof} Note that \autoref{lem:t-works} implies that the compact objects of $\C^{\heart}$ form an abelian category, so that (ii) $\Leftrightarrow$ (iii).
  
  Note that (i) implies (ii) by general non-sense, since $\C^{\heart}$ is the quotient of $\C_{\geq 0}$ by the essential image of $\C_{>0}$ and $i_{>0}$ was assumed to preserve compact objects.  
  
  It remains to show that (ii) implies (i): Note first that the objects of $\Coh(\C) \cap \C_{<0}$ are all compact in $\C_{<0}$ by \autoref{lem:t-works}.  Since $\C$ admits  all filtered colimits, and $\C_{<0}$ is closed under filtered colimits, there is a fully faithful functor
  \[ \Ind \left[ \Coh(\C) \cap \C_{<0} \right] \longrightarrow \C_{<0}. \]
  Notice that $\Coh(\C) \cap \C_{<0}$ admits all finite colimits, since this is true for $\Coh(\C)$ and $\Coh(\C)$ is preserved by $\tau_{<0}$.  Thus it is enough to prove that this functor is essentially surjective.  
  
  Since $\C$ is right complete, we are reduced to proving that the bounded objects are in the image.  Note that (ii) implies that the heart is in the image.  Both sides have finite homotopy limits, and the functor preserves them, so that considering the rotated fiber sequences
  \[ \tau_{\geq -k} \F \longrightarrow  (H_{-k} \F)[-k] \longrightarrow (\tau_{\geq -(k+1)} \F)[+1] \]
  shows that the image contains all bounded objects of $\C_{<-1}$ by induction on the range of non-vanishing homotopy groups.  Since the functor also preserves finite homotopy colimits, to complete the proof it suffices to show that $\Sigma \circ \Omega \isom \id$ on both sides. For then each object of $\C_{<0}$ will be a suspension of something in $\C_{<-1}$, which is in the image.  In each case, this follows because $\Coh(\C) \cap \C_{<0} = \Coh(\C)_{<0}$ and $\C_{<0}$ are the co-connective parts of a $t$-structure.
\end{proof}

This brings us to the following definition (which the previous Lemmas give various equivalent formulations and consequences of):
\begin{defn}\label{defn:t-coh} Suppose $\C$ is a stable $\infty$-category with $t$-structure.  We say that the $t$-structure is \demph{coherent}  if the following conditions are satisfied:
	\begin{itemize}
		\item The $t$-structure is compatible with filtered colimits;
		\item The $t$-structure is right complete;
                \item $\C^{\heart}$ is a locally coherent abelian category.
	\end{itemize}
\end{defn}

\subsection{Regular and complete $t$-structures}
\begin{defn}\label{defn:t-reg} Suppose $\C$ is a stable $\infty$-category with $t$-structure compatible with filtered colimits.  We have seen that $\C_{<0} \to \C_{<1}$, etc., preserves filtered colimits.  We say that the $t$-structure is (left) \demph{regular} if the natural map
  \[ \Fun^{\text{filtered colimits}}(\C, \D) = \ilim_n \Fun^{\text{filtered colimits}}((\C_{<n}, i_{<n}), \D) \]
  is an equivalence for every category $\D$ admiting filtered colimits.
\end{defn}

\begin{prop} Suppose that $\C$ is \emph{coherent}.  Then, $\C$ is regular if and only if $\C$ is compactly-generated by $\Coh(\C)$.  
  
  There is a universal regular $\infty$-category with $t$-structure mapping to $\C$, and it is given by the formula
  \[ \R(\C) \eqdef \Ind\left(\Coh(\C)\right) \longrightarrow \C \]
  The functor $\R(\C) \to \C$ preserves colimits, is $t$-exact, and the induced functor $\R(\C)_{<0} \to \C_{<0}$ is an equivalence.  The $t$-structure on $\R(\C)$ is also coherent.
\end{prop}
\begin{proof} 
  
  By hypothesis, $\C_{<0}$ is compactly generated with compact objects $\Coh(\C)_{<0}$.   The functors $\C_{<0} \to \C_{<1}$ preserve both filtered colimits and compact objects, so we see that
  \[ \colim_n^{\text{filtered colimits}} \C_{<n} = \Ind\left(\colim_n \Coh(\C)_{<n} \right) = \Ind\left( \Coh(\C) \right) \]
  In particular, the first colimit exists: This is the assertion that there is an $\infty$-category with the correct universal property; and it is given by the desired formula. 

Notice that the functor $\R(\C) \to \C$ preserves filtered colimits by construction, and finite colimits on compact objects by inspection, so that it preserves colimits.  Since both $t$-structures are compatible with filtered colimits, and since
\[ \R(\C)_{<0} = \Ind(\Coh(\C)_{<0}) \qquad \R(\C)_{>0} = \Ind(\Coh(\C)_{>0}) \] by construction, we see that this functor is $t$-exact.  It is evident that it induces an equivalence on co-connective objects.

  Let us verify that $\R(\C)$ is coherent: The $t$-structure is compatible with filtered colimits, as $\R(\C)_{<0} \to \R(\C)$ preserves filtered colimits by construction.  It is right complete and satisfies the extra coherent condition, since these both depend only on $\R(\C)_{<0} \isom \C_{<0}$.

\end{proof}

\begin{defn} Suppose $\C$ is a stable $\infty$-category with $t$-structure.  We say that the $t$-structure is (left) \demph{complete} if the natural functor
  \[ \C \to \ilim_n (\C_{<n}, \tau_{<n}) \]
  is an equivalence.\footnote{This is just a reformulation of the earlier definition!}
\end{defn}

\begin{prop} Suppose that $\C$ is a stable $\infty$-category with $t$-structure.  Then, there is a universal complete $\infty$-category with $t$-structure mapping to $\C$, and it is given by the formula
  \[ \C \longrightarrow \oh{\C} = \ilim_n \C_{<n} \]
  This functor is $t$-exact, and the induced functor $\C_{<0} \to \oh{\C}_{<0}$ is an equivalence.
  
  If $\C$ is coherent, then the $t$-structure on $\C$ is also coherent.
\end{prop}
\begin{proof} See \cite{LurieHA}*{$\S$1.2.1} for everything but the last sentence.
  
  Let us verify that $\oh{\C}$ is coherent if $\C$ is: The $t$-structure is compatible with filtered colimits since each functor in the inverse limit is so, and the other properties depend only on $\oh{\C}_{<0} \isom \C_{<0}$.
\end{proof}

The point of making these definitions is the following:
\begin{defn}\mbox{} \begin{enumerate}
  \item Let $\Coht_t$ denote the $\infty$-category whose objects are $\infty$-categories $\C$ with coherent $t$-structure; whose $1$-morphisms are colimit preserving and $t$-exact functors; and whose higher morphisms are as in $\Cat_\infty$.
  \item Let $\Reg_t \subset \Coht_t$ denote the subcategory whose objects are $\infty$-categories $\C$ with \emph{regular} $t$-structure.
  \item Let $\Cplt_t \subset \Coht_t$ denote the subcategory whose objects are $\infty$-categories $\C$ with \emph{complete} (and coherent) $t$-structure.
\end{enumerate}
\end{defn}

\begin{theorem}\label{thm:reg-coh-equiv} The composites
  \[ \C \mapsto \oh{\C} \colon \Reg_t \hookrightarrow \Coh_t  \longrightarrow \Cplt_t \]
  \[ \C \mapsto \R(\C) \colon \Cplt_t \hookrightarrow \Coh_t  \longrightarrow \Reg_t \]
  are inverse equivalences of $\infty$-categories.
\end{theorem}
\begin{proof} It follows from the above that the first functor is left-adjoint to the second.  It is enough to check that the unit and co-unit is an equivalence.  For instance if $\C \in \Reg_t$ then we must check that
  \[ \R(\oh{\C}) \longrightarrow \C \]
  is an equivalence.  Since both are regular and the functor is left $t$-exact and preserves filtered colimits, it is enough to note that it is an equivalence on co-connective objects, which we have seen.  The argument for the other adjoint is similar.
\end{proof}

\subsection{Tensor products, and functors, of $t$-structures}\label{ssec:t-tens}\mbox{}
\begin{defn} Suppose that $\C, \D$ are stable presentable $\infty$-categories with accessible $t$-structures compatible with filtered colimits.  Then, $\C \otimes \D$ is a stable presentable $\infty$-category, and we define an accessible $t$-structure compactible with filtered colimits on it by requiring that $(\C \otimes \D)_{\geq 0}$ be generated under colimits by objects of the form
  \[ c \otimes d : c \in \C_{\geq 0}, d \in \D_{\geq 0} \]

  Then, $(\C \otimes \D)_{<0}$ is characterized by being the right orthogonal to the above.
\end{defn}

Defining well-behaved $t$-structures on functor categories seems to be more subtle.  However, if $\C$ is compactly-generated we can get around this: 
\begin{defn} Suppose that $\C^c$ is a small stable idempotent-complete $\infty$-category with $t$-structure.  Then, $\C = \Ind(\C^c$) admits an accessible $t$-structure compatible with filtered colimits.    In this case, the dual of $\C$ identifies with $\Ind((\C^c)^{op})$ and this \emph{also} admits such a $t$-structure: Indeed, $(\C^c)^{op}$ admits a $t$-structure determined by setting
  \[ (\C^c)^{op}_{\geq 0} = (\C^c)_{\leq 0} \qquad (\C^c)^{op}_{\leq 0} = (\C^c)_{\geq 0} \]
  and simply ``reversing'' the truncation sequences.
\end{defn}

\begin{remark} If $\C$ is as in the previous definition, and $\D$ is a stable presentable $\infty$-category with $t$-structure, then we can again put a $t$-structure on $\Fun^L(\C, \D)$ as follows: Identify 
  \[ \Fun^L(\C, \D) \isom \Ind((\C^c)^{op}) \otimes \D \]
  and equip it with the $t$-structure from the previous two definitions.  Explicitly,
  \[ \Fun^L(\C,\D)_{\geq 0} \]
  is generated under colimits and extensions by objects of the form $\RHom_{\C}(c, -) \otimes d)$ for $c \in \C_{\geq 0}$ and $d \in \D_{\geq 0}$.

  A definition chase shows that $\Fun^L(\C,\D)_{<0}$ consists precisely of those $F$ for which 
  \[ \Map_{\Fun^L(\C,\D)}(\RHom_{\C}(c,-) \otimes d, F) = \Map_{\D}(d, F(c)) = \pt \]
  for all $c, d$ as above -- and this is precisely those $F$ which are left $t$-exact.
\end{remark}

\subsection{Quasi-coherent and ind-coherent complexes}
\begin{prop}\label{prop:qc-t-cplt} Suppose that $X$ is a geometric stack.  Then, $\QC(X)$ is a stable presentable $\infty$-category with accessible $t$-structure.  This $t$-structure is both left and right complete.  If $X$ is Noetherian, then $\QC(X)$ is coherent.
\end{prop}
\begin{proof} See DAG VIII for the first two sentences of the Proposition.  For the third, it is a classical statement that every object of $\QC(X)^{\heart}$ is a filtered colimit of its coherent subobjects.
\end{proof}

\begin{example}\label{ex:dcoh-cpltions} Suppose $A$ is a Noetherian ring.  Then, $A\mod$ carries a $t$-structure that is both left and right complete in the strong sense.  Meanwhile, the full subcategory $\DCoh A \subset A\mod$ carries a $t$-structure which is both left and right bounded.  In particular, it is weakly left and right complete, though not strongly so.   This fully faithful exact embedding into a left (resp., right) complete category identifies the left (resp., right) completion of $\DCoh A$ with full subcategories of $A\mod$:
  \begin{itemize}
      \item The left completion of $\DCoh A$ identifies with $\DCohp A$, the full-subcategory of modules $M$ with $H_i(M)$ coherent over $H_0(A)$ for all $i$ and vanishing for $i \ll 0$;
      \item The right completion of $\DCoh A$ identifies with $\DCohm A$, the full-subcategory of modules $M$ with $H_i(M)$ coherent over $H_0(A)$ for all $i$ and vanishing for $i \vinograd 0$;
      \item The left completion of the right completion (equivalently the other way around) of $\DCoh A$ identifies with $\DCohpm A$, the full sub-category of modules $M$ with $H_i(M)$ coherent over $H_0(A)$ for all $i$.
  \end{itemize}
  
  This provides an ``application'' of the formal symmetry of the definitions.  Suppose that $\omega \in A\mod$ is a dualizing complex.  This means that $\omega$ has homologically bounded above coherent homology, finite injective dimension, and the natural map $A \to \RHom_A(\omega,\omega)$ is an equivalence.  It follows that the induced duality functor 
  \[ \DD = \RHom_A(-,\omega) \colon \DCoh A^{op} \longrightarrow \DCoh A \]
  is an equivalence and that it is left and right $t$-exact up to finite shifts (where the opposite category gets the opposite $t$-structure).  By formal nonsense, it induces an equivalence on left completion of right completions
  \[ \DD \colon \left( \DCohpm A\right)^{op} \isom \DCohpm A\]
\end{example}

\begin{prop}\label{prop:groth-duality} Suppose that $X$ is a geometric stack of finite type over $S = \Spec R$ for $R$ a Noetherian ring admitting a dualizing complex $\omega_R$.  Then, 
  \begin{itemize}
      \item $X$ admits a dualizing complex, i.e., an $\omega_X \in \QC(X)$ such that for any smooth map $U = \Spec A \to X$ the restriction $\res{\omega_X}{U}\in A\mod$ is a dualizing complex in the above sense;
    \end{itemize}

    Suppose that $X$ is a geometric stack admitting a dualizing complex $\omega_X$, and that $X$ has finite cohomological dimension.  Then,
    \begin{enumerate}
      \item The formation of $\DD(-)=\RHom_X(-, \omega_X)$ is smooth local on $X$ for $- \in \DCohpm X$;
      \item $\DD$ induces an anti-equivalence 
        \[ \DD \colon \DCohpm X^{op} \isom \DCohpm X \]
        which is left and right $t$-exact up to finite shifts.  In particular, it interchanges the bounded above and bounded below complexes.
     
      \item The fully faithful, $t$-exact, embedding
      \[ i \colon \DCoh X \to \DCohp X \]
      exhibits $\DCohp X$ as a left $t$-completion of $\DCoh X$.

     \item Grothendieck duality determines a fully faithful embedding
        \[ \DD \colon (\DCoh X)^{op} \to \DCohp X \]
        which is left and right $t$-exact up to a shift.  It exhibits $\DCohp X$ as a left $t$-completion of $(\DCoh X)^{op}$ up to finite shifts.
  \end{enumerate}
\end{prop}
\begin{proof} We first prove the existence of a dualizing complex:
  \begin{itemize}
    \item 
Note that the notion of dualizing complexes is smooth local on affine rings.

Furthermore, one can show that if $\omega$ and $\omega'$ are two dualizing complexes on $U = \Spec A$, then $\RHom_A(\omega, \omega')$ is a graded line on $A$.  So, the $\infty$-category of dualizing complexes is equivalent to a $1$-category, and the (ordinary) stack of dualizing complexes on $X$ form a torsor -- on the smooth site of $X$ -- for the Picard groupoid of graded lines $\ZZ \ltimes B\GG_m$.  Call this torsor $\Dualiz^*_X \to X$.

Let $\det \LL_{/X}$ be the graded line, on the smooth site of $X$, given by $U \mapsto \det \LL_{U/X}$.  The existence of the functor $f^!$ on dualizing complexes (for finite type maps of Noetherian rings) implies that there is a $\ZZ \ltimes B\GG_m$-torsor $\Dualiz^!_S$ on the fppf site of $S$.  The natural isomorphisms $f^* \otimes \det \LL_f \isom f^!$ for smooth maps provide isomorphisms of $\ZZ \ltimes B\GG_m$-torsors
shows that $\det\LL_{/X}$ and $\det \LL_{/S}$ give isomorphisms
\[ \Dualiz^*_X \stackrel{\det\LL_{/X}}\longrightarrow \res{\Dualiz^*_S}{X} \stackrel{\det\LL_{/S}}\longleftarrow \res{(\det \LL_{/S} \otimes \Dualiz^*_S)}{X} \]

Since $\Dualiz^*_S$ was trivial by assumption, we are done.
\end{itemize}

Now the rest:
  \begin{enumerate}
    \item Let $\C \subset \QC X$ denote the full subcategory on which $\DD(-)$ is smooth local.  Note that $\Perf X \subset \C$ by dualizability.
      
      Since $\omega_X$ is bounded above, a convergence result implies that $\DCohp X \subset \C$: It suffices to show that for each $i$, $H_i \DD(\tau_{\leq k} -)$ is constant for $k \geq N(i)$ with $N(i)$ depending only on $i$ and the boundedness of $\omega_X$, and in particular remaining true with the same constant after smooth base change.  This implies that $\DCohp X \subset \C$, since for $\F \in \DCohp X$ and each $k$ there exists a perfect complex $\P_k$ and a map $\P_k \to \F$ inducing an equivalence on $\tau_{\leq k}$.
      
      Similarly, since $\omega_X$ has finite injective dimension and $X$ has finite cohomological amplitude, we see that $H_i \DD(\tau_{\geq -k} -)$ is constant for $k \geq M(i)$ with $M(i)$ depending only on $i$, the injective dimension of $\omega_X$ and the cohomological amplitude of $X$.  This implies that $\DCohpm X \subset \C$, since if $\F \in \DCohpm X$ then $\tau_{\geq -k} \F \in \DCohp X$ for each $k$.
      \item In light of (ii), the claim is smooth local so we may suppose that $X = \Spec A$.  Let $\C \subset \DCohpm X$ denote the full subcategory on which the double duality map $\F \to \DD \circ \DD(\F)$ is an equivalence.  Since $\omega$ is a dualizing complex, $A \in \C$ and so $\Perf A \subset \C$.

        The proof of (i), together with the fact that $A\mod$ is left and right complete, shows that $\DD$ is left and right $t$-exact up to a shift.  Consequently, so is $\DD \circ \DD$.  This implies that both $H_i(\tau_{leq k} \F)$ and $H_i(\DD \circ \DD (\tau_{\leq k}\F))$ eventually stabilize.  This implies that $\DCohp A \subset \C$ by the same approximation-by-perfect argument as above.  Similarly, $H_i(\tau_{geq -k} \F)$ and $H_i(\DD \circ \DD (\tau_{\geq -k}\F))$ eventually stabilize.  This implies that $\DCohpm A \subset \C$ by approximation by pseudo coherent (or almost perfect) complexes.

        This proves that $\DD$ is an equivalence.  The proof showed that it was left and right $t$-exact up to a shift.
   
        \item By definition $\DCoh X$ consists of the left bounded objects of $\DCohp X$, so it suffices to note that Postnikov towers in $\DCohp X$ are convergent since they are so in $\QC X$.
        \item This follows from (ii) and (iii). More directly, one sees that the left $t$-completion of $\DCoh X^{op}$ identifies with $(\DCohm X)^{op}$ and (ii) shows that $\DD$ identifies this with $\DCohp X$. \qedhere
  \end{enumerate}
\end{proof}

\begin{remark} It seems not entirely clear that the existence of $f^!$ is written down for derived stacks not of finite type over a char. $0$ field.  
  
  Let us note that if $X$ and $S$ are \emph{classical}, then the classical statements -- at the level of derived categories -- suffice for our purposes.  Indeed, if $X$ and $S$ are classical, then $\Dualiz^*_X$ is equivalent to a (classical) $1$-groupoid and embeds fully faithfully into the maximal subgroupoid of the derived category of $X$.  So, functoriality at the level of derived categories suffices.

\end{remark}

\begin{prop}\label{prop:qcsh-t-reg} Suppose that $X$ is a geometric stack of finite type over $\Spec k$ for $k$ a characteristic zero field.  Then, $\QCsh(X)$ is a stable presentable $\infty$-category with accessible $t$-structure.  This $t$-structure is coherent and regular.
  
  Furthermore, the natural map
  \[ \QCsh(X) \to \QC(X) \]
  realizes $\QCsh(X)$ as the regularization of $\QC(X)$, and $\QC(X)$ as the completion of $\QCsh(X)$.
\end{prop}
\begin{proof} See \cite{DrinfeldGaitsgory}:  One uses a finite-length stratification by global quotient stacks to show that $X$ has finite cohomological dimension (this is where one uses characteristic zero); from this, we deduce that $\QCsh(X)^c = \DCoh(X)$.  Then, one uses the stratification to show that $\QCsh(X)^{\heart}$ generates, reducing to the statement about ordinary quasi-coherent sheaves being unions of their coherent subsheaves.
\end{proof}


\section{Correction to Lemma 3.0.13}\label{s:correction}
We thank Johan de Jong and Noah Olander for questioning the validity of Lemma 3.0.13. Indeed,
it is false; we thank Germ\'{a}n Stefanich for providing a counterexample. In place of  Lemma 3.0.13, one can use the following. We thank  Germ\'{a}n Stefanich for generously providing the argument.

\begin{prop}[\cite{Luriespec} Proposition 9.6.3.1]\label{prop single object}
Let $X$ be a quasi-compact, quasi-separated algebraic space. Then there exists a perfect complex $\mathcal{F}$ on $X$ such that the smallest subcategory of $\QC(X)$ containing $\mathcal{F}$ and closed under colimits and extensions contains $\QC(X)_{\geq 0}$.
\end{prop}

\begin{prop}\label{prop correction}
Let $X$ be a quasi-compact separated algebraic space over an affine $S$. If $G$ is a perfect complex generating $\QC(X)$ then
\[
\Hom_X(G, -): \QC(X) \rightarrow \QC(S)
\]
detects the properties of being bounded above and being bounded below.
\end{prop}
\begin{proof}

Let $M$ be an object of $\QC(X)$ such that $\Hom_X(G, M)$ is bounded above. Since $G$ is a compact generator, we have that $\Hom_X(\mathcal{F}, M)$ is bounded above for the compact  object $\mathcal{F}$ from Proposition \ref{prop single object}. Assume that $\Hom_X(\mathcal{F}, M)$ is bounded above by $i$. Then $\Hom_X(H, M)$ is bounded above by $i$ for any object $H$ in the closure of $\mathcal{F}$ under colimits and extensions. In particular, this holds for any connective object $H$. It follows that $M$ is bounded above by $i$.

Assume now given an object $M$ of $\QC(X)$ such that $\Hom_X(G, M)$ is bounded below. Recall that the t-structure on $\QC(X)$ is defined in such a way that being connective can be detected by pullback to affine schemes. We therefore need to show that there exists a number $j$ such that for every map $f: T \rightarrow X$ from an affine scheme $T$ into $X$, the pullback $f^*M$ is bounded below by $j$. This is the same as $\Gamma(T, f^*M)$ being bounded below by $j$. We compute this as 
\[
\Gamma(X, f_*f^*M) = \Gamma(X, f_*\mathcal{O}_{T} \otimes M).
\] 
The map $f$ is affine schematic, and therefore $f_*$ is t-exact. This implies that $f_*\mathcal{O}_T$ is connective and hence $f_*\mathcal{O}_{T} \otimes M$ is contained in the closure of $\mathcal{F} \otimes M $ under extensions and colimits, where $\mathcal{F}$ is as in Proposition \ref{prop single object}. We conclude that $\Gamma(T, f^*M)$ is bounded below by $j$, where $j$ is the lower bound for $\Gamma(X, \mathcal{F} \otimes M) = \Hom_X(\mathcal{F}^*, M)$ (which exists since $\mathcal{F}^*$ is in the closure of $G$ under finite colimits, shifts, and retracts). It follows that $M$ itself is bounded below by $j$.
\end{proof}

\begin{remark}
Note that with minor modification we can weaken ``separated'' to ``quasi-separated'' in the previous Proposition as follows. In the proof of the ``bounded below'' case, we note that it suffices to verify that $f^* M$ is bounded below by $j$ for a single faithfully-flat cover by an affine scheme $f: T \rightarrow X$; and that it suffices for the rest of the argument to show that $f_* \mathcal{O}_T \in \QC(X)_{\geq -d}$ for some $d$.  This follows from $f$ having finite cohomological dimension, for concreteness: we can check this on the affine cover $f$, so it is enough to check $\Gamma(f^* f_* \mathcal{O}_T)$ is bounded below by $-d$; by base-change, this is $\Gamma(T \times_X T, \mathcal O_{T \times_X T})$ and so bounded below by $-d$ for some $d$ (e.g., the Zariski dimension of $T \times_X T$).
\end{remark}

\bibliography{coherent}

\end{document}

%% file: macros.tex
\usepackage{mathrsfs}
\usepackage[mathcal]{eucal}
\usepackage{microtype}

\usepackage{fullpage}

\usepackage{textcmds}  
\usepackage{amsfonts, amssymb, amsmath, amsthm, amsopn, bm, latexsym, bbm}
\usepackage[dvipsnames,usenames]{color}
\usepackage{graphicx, fancyhdr, fancybox, ifthen, enumitem}
\usepackage[ps,matrix,curve,frame,arrow,rotate,line]{xy}

\usepackage{xr-hyper}

\usepackage{aliascnt} 
\usepackage[colorlinks=true,linkcolor=blue,citecolor=blue,urlcolor=blue,citebordercolor={0 0 1},urlbordercolor={0 0 1},linkbordercolor={0 0 1}]{hyperref} 
\usepackage[alphabetic]{amsrefs} 

\renewcommand{\eprint}[1]{\href{https://urldefense.com/v3/__http://arxiv.org/abs/*1*7D*7Barxiv:*1__;IyUlIw!!GSt_xZU7050wKg!5e2VFGOgXR6qNFqqDgCkkknu5OrXS-YtQKIFAF9ewZ_5pbwzi8zJTKx7pd3ES5jBWcc$ }} 

\let\oldfootnotemark\footnotemark
\let\oldfootnotetext\footnotetext
\let\oldfootnote\footnote
\renewcommand\footnote[1]{\addtocounter{footnote}{1}\hypertarget{fnbackref.\arabic{footnote}}{}\addtocounter{footnote}{-1}\oldfootnote{#1\fnbackref}}
\renewcommand\footnotemark{\addtocounter{footnote}{1}\hypertarget{fnbackref.\arabic{footnote}}{}\addtocounter{footnote}{-1}\oldfootnotemark}
\renewcommand\footnotetext[1]{\oldfootnotetext{#1\fnbackref}}
\newcommand{\fnbackref}{\hyperlink{fnbackref.\arabic{footnote}}{\footnotesize$\uparrow$}}

\definecolor{mydefnblue}{rgb}{.2,.2,.7}
\newcommand{\demph}[1]{\textcolor{mydefnblue}{\it #1}}

\makeatletter
\newcommand{\sss}{\@startsection{subsubsection}{2}{0pt}{-3ex
plus -1ex minus -0.2ex}{-2mm plus -0pt minus
-2pt}{\normalfont\bfseries}} 
\makeatother

\definecolor{mynotecol}{rgb}{.4,.1,.1}
\let\oldmarginpar\marginpar
\renewcommand\marginpar[1]{\mbox{}\oldmarginpar[\raggedleft\hspace{0pt}\textcolor{mynotecol}{\small #1}]%
{\raggedright\hspace{0pt}\textcolor{mynotecol}{\small #1}}}

\newcommand{\refnewtheoremn}[4]{%
\newaliascnt{#1}{#2}
\newtheorem{#1}[#1]{#3}
\aliascntresetthe{#1}
\expandafter\providecommand\csname #1autorefname\endcsname{#4}}

\newcommand{\refnewtheorem}[3]{\refnewtheoremn{#1}{#2}{#3}{#3}}

\theoremstyle{plain}
\newtheorem{theorem}{Theorem}[subsection]

\refnewtheorem{conjecture}{theorem}{Conjecture}
\refnewtheoremn{prop}{theorem}{Proposition}{Prop.}
\refnewtheorem{lemma}{theorem}{Lemma}
\refnewtheoremn{corollary}{theorem}{Corollary}{Cor.}
\refnewtheorem{claim}{theorem}{Claim}

\theoremstyle{definition}
\refnewtheoremn{defn}{theorem}{Definition}{Def.}
\refnewtheorem{constr}{theorem}{Construction}
\refnewtheorem{example}{theorem}{Example}
\refnewtheorem{computation}{theorem}{Computation}
\refnewtheorem{notation}{theorem}{Notation}

\refnewtheorem{remark}{theorem}{Remark}
\refnewtheorem{spec}{theorem}{Speculation}

\refnewtheoremn{na}{theorem}{}{\!\!}

	{\begin{Sbox}\begin{minipage}{\linewidth}\begin{conjecture}}%
	{\end{conjecture}\end{minipage}\end{Sbox}\par\vspace{2pt}\noindent\fbox{\vbox{\vspace{.25pt}\TheSbox\vspace{.25pt}}}}

	{\begin{Sbox}\begin{minipage}{\linewidth}\begin{theorem}}%
	{\end{theorem}\end{minipage}\end{Sbox}\par\vspace{2pt}\noindent\fbox{\vbox{\vspace{.25pt}\TheSbox\vspace{.25pt}}}}

	{\begin{Sbox}\begin{minipage}{\linewidth}\begin{defn}}%
	{\end{defn}\end{minipage}\end{Sbox}\par\vspace{2pt}\noindent\fbox{\vbox{\vspace{.25pt}\TheSbox\vspace{.25pt}}}}

\newcommand{\heart}{\heartsuit}

\newcommand{\eqdef}{\overset{\text{\tiny def}}{=}}
\newcommand{\ilim}{\mathop{\varprojlim}\limits}
\newcommand{\dlim}{\mathop{\varinjlim}\limits}

\newcommand{\isom}{\simeq}           

\newcommand{\ps}[1]{[\![#1]\!]}
\newcommand{\pl}[1]{(\!(#1)\!)}
\newcommand{\sh}[1]{#1^{!}}
\newcommand{\dual}{{\mkern-1.5mu{}^{\vee}}}

\newcommand{\shotimes}{\overset{!}{\otimes}}

\newcommand{\ol}[1]{\overline{#1}}

\newcommand{\oh}[1]{\widehat{#1}}
\newcommand{\sq}[1]{\widetilde{#1}}

\newcommand{\res}[2]{\left. #1 \right|_{#2}}

\newcommand{\DD}{\mathbb{D}}
\newcommand{\GG}{\mathbb{G}}

\newcommand{\LL}{\mathbb{L}}

\newcommand{\PP}{\mathbb{P}}

\newcommand{\ZZ}{\mathbb{Z}}

\let\vinograd\gg 
\renewcommand{\gg}{\mathfrak{g}}

\newcommand{\E}{\mathscr{E}}

\newcommand{\F}{\mathscr{F}}
\newcommand{\G}{\mathscr{G}}

\newcommand{\K}{\mathscr{K}}

\renewcommand{\P}{\mathscr{P}}

\newcommand{\X}{\mathscr{X}}
\newcommand{\Y}{\mathscr{Y}}
\newcommand{\Z}{\mathscr{Z}}

\newcommand{\C}{\mathcal{C}}
\newcommand{\D}{\mathcal{D}}
\renewcommand{\H}{\mathcal{H}}

\renewcommand{\O}{\mathcal{O}}

\newcommand{\pt}{\ensuremath{\text{pt}}}

\newcommand{\op}{{\mkern-1mu\scriptstyle{\mathrm{op}}}}

\newcommand{\red}{{\mkern-1.5mu\scriptstyle{\mathrm{red}}}}

\renewcommand{\mod}{\text{\rm{-mod}}}  

\DeclareMathOperator{\Coh}{Coh}
\DeclareMathOperator{\DCoh}{DCoh}

\DeclareMathOperator{\Perf}{Perf}
\DeclareMathOperator{\QC}{QC}
\newcommand{\QCsh}{\sh{\QC}}

\newcommand{\Phish}{\sh{\Phi}}

\DeclareMathOperator{\tr}{tr}

\DeclareMathOperator{\cone}{cone}

\DeclareMathOperator{\supp}{supp}

\DeclareMathOperator{\Hom}{Hom}
\DeclareMathOperator{\Map}{Map}

\DeclareMathOperator{\RGamma}{R\Gamma}

\DeclareMathOperator{\Sym}{Sym}

\DeclareMathOperator{\RHom}{RHom}

\DeclareMathOperator{\Spec}{Spec}

\DeclareMathOperator{\Fun}{Fun}

\DeclareMathOperator{\id}{id}

\DeclareMathOperator{\Ind}{Ind}

\DeclareMathOperator{\fib}{fib}

\DeclareMathOperator{\Tot}{Tot}
\newcommand{\colim}{\operatorname{colim}\limits}

\newcommand{\dgcatidm}{\dgcat^{\mkern-2mu\scriptstyle{\mathrm{idm}}}}
\newcommand{\dgcatbig}{\dgcat^{\mkern-2mu\scriptstyle{\infty}}}
\newcommand{\dgcat}{\mathbf{dgcat}}

\newcommand{\Cat}{\mathbf{Cat}}
\newcommand{\sSet}{\mathbf{sSet}}

\newcommand{\Alg}{\mathbf{Alg}}

\newcommand{\adjunct}[2]{\xymatrix@1{ #1 \ar@<.7ex>[r] & \ar@<.7ex>[l] #2 }}

\newcommand{\ssetr}[2]{\xymatrix@1{ #1 \ar[r] & \ar@<.7ex>[l] #2 \ar@<-.7ex>[l] \ar@<.7ex>[r] \ar@<-.7ex>[r] & \cdots \ar@<1.4ex>[l] \ar@<-1.4ex>[l] \ar[l] }}

\newcommand{\ssetlar}[4]{\xymatrix@1{  #1 \ar@<.7ex>[r]^-{#2} \ar@<-.7ex>[r]_-{#3} & #4}}

\newcommand{\ssetl}[2]{\xymatrix@1{ \cdots \ar@<1.4ex>[r] \ar@<-1.4ex>[r] \ar[r] & #2 \ar@<.7ex>[r] \ar@<-.7ex>[r] \ar@<.7ex>[l] \ar@<-.7ex>[l] & #1 \ar[l]}}

\newcommand{\ssetrr}[3]{\xymatrix@1{ #1 \ar[r] & #2 \ar@<.7ex>[l] \ar@<-.7ex>[l] \ar@<.7ex>[r] \ar@<-.7ex>[r] & #3 \ar@<1.4ex>[l] \ar@<-1.4ex>[l] \ar[l]  \ar@<1.4ex>[r] \ar@<-1.4ex>[r] \ar[r] & \cdots\ar@<.7ex>[l] \ar@<-.7ex>[l] \ar@<2.1ex>[l] \ar@<-2.1ex>[l]}}

\newcommand{\ssetll}[3]{\xymatrix@1{ \cdots\ar@<.7ex>[r] \ar@<-.7ex>[r] \ar@<2.1ex>[r] \ar@<-2.1ex>[r] & #3 \ar@<1.4ex>[r] \ar@<-1.4ex>[r] \ar[r]  \ar@<1.4ex>[l] \ar@<-1.4ex>[l] \ar[l] & #2 \ar@<.7ex>[r] \ar@<-.7ex>[r] \ar@<.7ex>[l] \ar@<-.7ex>[l] & #1 \ar[l]}}

%% file: coherent.bbl
\begin{bibdiv}
\begin{biblist}

\bib{BLL}{article}{
      author={Bondal, Alexey~I.},
      author={Larsen, Michael},
      author={Lunts, Valery~A.},
       title={Grothendieck ring of pretriangulated categories},
        date={2004},
        ISSN={1073-7928},
     journal={Int. Math. Res. Not.},
      number={29},
       pages={1461\ndash 1495},
      eprint={math/0401009},
}

\bib{BvdB}{article}{
      author={Bondal, Alexey},
      author={van~den Bergh, Michel},
       title={Generators and representability of functors in commutative and
  noncommu- tative geometry.},
        date={2003},
     journal={Mosc. Math. J. 3},
      number={1},
       pages={1\ndash 36, 258},
      eprint={math/0204218},
}

\bib{BFN}{article}{
      author={Ben-Zvi, David},
      author={Francis, John},
      author={Nadler, David},
       title={Integral transforms and {D}rinfeld centers in derived algebraic
  geometry},
        date={2010},
        ISSN={0894-0347},
     journal={J. Amer. Math. Soc.},
      volume={23},
      number={4},
       pages={909\ndash 966},
      eprint={0805.0157},
}

\bib{BNglue}{article}{
      author={Ben-Zvi, D.},
      author={},
      author={Nadler, D.},
       title={Betti spectral gluing},
      eprint={1602.07379},
}

\bib{BNP2}{article}{
      author={Ben-Zvi, D.},
      author={},
      author={Nadler, D.},
      author={Preygel, A.},
       title={A spectral incarnation of affine character sheaves},
      eprint={1312.7163},
}

\bib{CLO-Nagata}{article}{
      author={Conrad, Brian},
      author={Lieblich, Max},
      author={Olsson, Martin},
       title={Nagata compactification for algebraic spaces},
}

\bib{DrinfeldGaitsgory}{article}{
      author={Drinfeld, V.},
      author={Gaitsgory, D.},
       title={On some finiteness questions for algebraic stacks},
}

\bib{FrenkelGaitsgory-Dmod}{article}{
      author={{Frenkel}, E.},
      author={{Gaitsgory}, D.},
       title={{D-modules on the affine flag variety and representations of
  affine Kac-Moody algebras}},
        date={2009},
     journal={Represent. Theory},
      volume={13},
       pages={470\ndash 608},
      eprint={0712.0788},
}

\bib{indcoh}{article}{
      author={{Gaitsgory}, Dennis},
       title={Notes on geometric langlands: Ind-coherent sheaves},
        date={2011-05},
      eprint={1105.4857},
}

\bib{illusie}{article}{
      author={{Illusie}, Luc},
       title={G«en«eralit«es sur les conditions de Þnitude dans les cat«egories
  d«eriv«ees: Th«eorie des intersections et th«eor`eme de riemann-roch (sga
  6)},
        date={1971},
     journal={Lecture Notes in Math.},
      number={225},
       pages={78\ndash Ð296},
}

\bib{Knutson-AlgSp}{article}{
      author={Knutson, D.},
       title={Algebraic spaces},
}

\bib{LO}{article}{
      author={Lunts, V.},
      author={Orlov, D.},
       title={Uniqueness of enhancement for triangulated categories},
        date={2010},
     journal={J. Amer. Math. Soc.},
      volume={23},
      number={3},
       pages={853\ndash 908},
}

\bib{lowrey}{article}{
      author={{Lowrey}, Parker},
       title={The moduli stack and motivic hall algebra for the bounded derived
  category},
      eprint={1110.5117},
}

\bib{DAG-XI}{article}{
      author={Lurie, Jacob},
       title={Derived algebraic geometry {XI}: Descent theorems},
}

\bib{DAG-XII}{article}{
      author={Lurie, Jacob},
       title={Derived algebraic geometry {XII}: Proper morphisms, completions,
  and the {G}rothendieck existence theorem},
}

\bib{LurieHA}{book}{
      author={Lurie, Jacob},
       title={Higher algebra},
  note={\url{https://urldefense.com/v3/__http://math.harvard.edu/*lurie/papers/HigherAlgebra.pdf__;fg!!GSt_xZU7050wKg!5e2VFGOgXR6qNFqqDgCkkknu5OrXS-YtQKIFAF9ewZ_5pbwzi8zJTKx7pd3E9NpaP8g$
  }},
}

\bib{Luriespec}{article}{
      author={Lurie, Jacob},
       title={Spectral algebraic geometry},
  note={\url{https://urldefense.com/v3/__https://www.math.ias.edu/*lurie/papers/SAG-rootfile.pdf__;fg!!GSt_xZU7050wKg!5e2VFGOgXR6qNFqqDgCkkknu5OrXS-YtQKIFAF9ewZ_5pbwzi8zJTKx7pd3EYixUACQ$
  }},
}

\bib{Neeman}{article}{
      author={Neeman, A.},
       title={Stable homotopy as a triangulated functor},
        date={1992},
     journal={Invent. Math.},
      volume={109},
       pages={17\ndash 40},
}

\bib{Orlov}{article}{
      author={Orlov, D.~O.},
       title={Equivalences of derived categories and {$K3$} surfaces},
        date={1997},
        ISSN={1072-3374},
     journal={J. Math. Sci. (New York)},
      volume={84},
      number={5},
       pages={1361\ndash 1381},
        note={Algebraic geometry, 7},
}

\bib{toly}{article}{
      author={Preygel, Anatoly},
       title={Thom-{S}ebastiani \& duality for matrix factorizations},
        date={2011-01},
     journal={arXiv e-print},
      eprint={1101.5834},
}

\bib{toly-loops}{article}{
      author={Preygel, Anatoly},
       title={Loop spaces and connections, revisited},
        date={2014},
}

\bib{RvdB}{article}{
      author={Rizzardo, A.},
      author={den Bergh, M.~Van},
       title={An example of a non-fourier-mukai functor between derived
  categories of coherent sheaves},
      eprint={1410.4039},
}

\bib{Toen}{article}{
      author={To{\"e}n, Bertrand},
       title={The homotopy theory of {$dg$}-categories and derived {M}orita
  theory},
        date={2007},
        ISSN={0020-9910},
     journal={Invent. Math.},
      volume={167},
      number={3},
       pages={615\ndash 667},
      eprint={math/0408337},
}

\bib{TT}{article}{
      author={Thomason, R.},
      author={Trobaugh, T.},
       title={Higher algebraic k-theory of schemes and of derived categories},
        date={1990},
     journal={Progr. Math.},
      number={88},
       pages={247\ndash 435},
        note={The Grothendieck Festschrift III},
}

\end{biblist}
\end{bibdiv}
